\documentclass[a4paper, 11pt]{amsart}
\usepackage{amsmath, amsthm, amscd, amssymb, amsfonts, amsxtra, amssymb, latexsym}
\usepackage{enumerate}
\usepackage{verbatim}
\usepackage{float}
\usepackage{pb-diagram}
\usepackage{array}

\usepackage{hyperref}
\hypersetup{colorlinks = true,	allcolors  = blue}

\newcommand{\Z}{\mathbb{Z}}
\newcommand{\N}{\mathbb{N}}
\newcommand{\ff}{\mathbb{F}}
\newcommand{\Tr}{\operatorname{Tr}}

\newcommand{\sk}{\smallskip}
\newcommand{\msk}{\medskip}

\newtheorem{thm}{Theorem}[section]
\newtheorem{prop}[thm]{Proposition}
\newtheorem{lem}[thm]{Lemma}
\newtheorem{coro}[thm]{Corollary}
\theoremstyle{definition}
\newtheorem{rem}[thm]{Remark}
\newtheorem{exam}[thm]{Example}
\newtheorem{defi}[thm]{Definition}

\theoremstyle{remark}
\newtheorem{note}[thm]{Note}
\usepackage{color}
    
\def\blue{\color{blue}}

\def\black{\color{black}}

\begin{document} \sloppy
\numberwithin{equation}{section}
\title[Spectral properties of GP-graphs of ($q^\ell+1$)-th powers]{Spectral properties of generalized Paley graphs \\ of ($q^\ell+1$)-th powers and applications}
\author{Ricardo A.\@ Podest\'a, Denis E.\@ Videla}
\dedicatory{\today}
\thanks{Partially supported by CONICET, FONCYT and SECyT-UNC}
\address{Ricardo A.\@ Podest\'a. FaMAF -- CIEM (CONICET), Universidad Nacional de C\'ordoba. 
	\newline Av.\@ Medina Allende 2144, Ciudad Universitaria, (5000), C\'ordoba, Argentina. \newline
	{\it E-mail: podesta@famaf.unc.edu.ar}}
\address{Denis E.\@ Videla. FaMAF -- CIEM (CONICET), Universidad Nacional de C\'ordoba. 
	\newline	Av.\@ Medina Allende 2144, Ciudad Universitaria, (5000), C\'ordoba, Argentina.
	\newline {\it E-mail: devidela@famaf.unc.edu.ar}}

\begin{abstract}
We consider a special class of generalized Paley graphs over finite fields, namely the Cayley graphs 
with vertex set $\ff_{q^m}$ and connection set the nonzero $(q^\ell+1)$-th powers in $\ff_{q^m}$, as well as their complements. 
We explicitly compute the spectrum and the energy of these graphs. As a consequence, the graphs turn out to be (with trivial exceptions) simple, connected, non-bipartite, integral and strongly regular, of pseudo or negative Latin square type. 
Using the spectral information we compute several invariants of these graphs. We exhibit infinitely many pairs of integral equienergetic non-isospectral graphs. 
As applications, on the one hand we solve Waring's problem over $\ff_{q^m}$ for the exponents 
$q^\ell+1$, for each $q$ and for infinitely many values of $\ell$ and $m$. We obtain that the Waring's number $g(q^\ell+1,q^m)=1$ or $2$, depending on $m$ and $\ell$, thus solving some open cases. On the other hand, we construct infinite towers of integral Ramanujan graphs in all characteristics. Finally, we give the Ihara zeta functions of these graphs. 
\end{abstract}

\maketitle

\section{Introduction} \label{sec:1}
Given a graph $\Gamma$ of $n$ vertices, $\lambda$ is an eigenvalue of $\Gamma$ if it is an eigenvalue of its 
adjacency matrix $A$. The \textit{spectrum} of the graph $\Gamma$ is the spectrum of $A$ counted with multiplicities 
(a multiset). Let $\lambda_1 \ge \cdots \ge \lambda_n$ be the eigenvalues of $\Gamma$ and let $d_i=m(\lambda_i)$ denotes the multiplicity of $\lambda_i$ for each $i=1,\ldots, n$. If $n=n_\Gamma$ is the number of distinct eigenvalues, it is usual to denote the spectrum by 
$$Spec(\Gamma) = \{ [\lambda_1]^{d_1}, [\lambda_2]^{d_2}, \ldots, [\lambda_n]^{d_n}\}$$ 
now with $\lambda_1 > \lambda_2 > \cdots > \lambda_n$. 
Two graphs are said \textit{isospectral} if they have the same spectrum.
Some structural properties of a graph $\Gamma$ can be read from its spectrum. In fact, if $\Gamma$ is $k$-regular then $\lambda_1=k$. Furthermore, $\Gamma$ is connected if and only if $d_1 =1$ and it is bipartite if and only if $\lambda_n=-k$. 
We will denote by $\bar \lambda_i$ and $\bar d_i = m(\bar \lambda_i)$ the eigenvalues and multiplicities of the complementary graph $\bar \Gamma$, respectively.
A graph $\Gamma$ is said to be \textit{integral} if $Spec(\Gamma)\subset \mathbb{Z}$. Generically, there are few integral graphs. In fact, the set of all $k$-regular connected integral graphs with fixed $k$ is finite for every $k$ (\cite{Cv}). 
A $k$-regular undirected connected graph is said to be 
\textit{Ramanujan} if
\begin{equation} \label{rama}
\lambda(\Gamma) := \max\limits_{1 \le i \le n} \{|\lambda_i|:  |\lambda_i| \ne k \} \le 2\sqrt{k-1}. 
\end{equation}
We remark that from a non-bipartite integral Ramanujan graph $\Gamma$, one can obtain a bipartite integral Ramanujan graph simply by considering the ``double'' of $\Gamma$, namely $\Gamma \otimes K_2$.

The paper has three main goals: ($a$) to introduce an interesting family of graphs and to study their structural and spectral properties in detail; ($b$) to produce pairs of integral equienergetic non-isospectral integral graphs; and ($c$) to construct infinite families of integral, non-bipartite, Ramanujan graphs. We will achieve these goals by considering a subclass of semiprimitive generalized Paley graphs over arbitrary finite fields. They will also have the extra property of being strongly regular graphs.

\subsubsection*{Outline and results.}
In Section \ref{sec:2} we consider a family of Cayley graphs $\Gamma_{q,m} (\ell) = X(\ff_{q^m},S_\ell)$ 
where $S_\ell$ is the set of nonzero $(q^\ell+1)$-th powers in $\ff_{q^m}^*$ with $\ell, m \in \Z$ satisfying $0 \le \ell <m$. 
They form a subclass of the generalized Paley graphs $X(\ff_{q}, \{x^k : x \in \ff_q\})$ (GP-graphs for short). 
In fact, they are a subfamily of the so called semiprimitive GP-graphs. When $m_\ell=\tfrac{m}{(m,\ell)}$ is odd, the family includes complete graphs $K_{q^m}$ when $q$ is even and classic Paley graphs $P(q^m)$ when $q$ is odd. 
In the case $m_\ell$ even, 
the graphs turn out to be generalized Paley graphs (as defined in \cite{LP}), with particular parameters. 
We focus on the case $m_\ell$ even and show that it is enough to take $\ell \mid m$. Hence, we consider the family 
$$\mathcal{G}_{q,m} = \{\Gamma_{q,m}(\ell) \, : \, 1 \le \ell \le \tfrac m2, \, \ell \mid m, \, \text{and } m_\ell=\tfrac{m}{\ell} \text{ even}\}.$$
We determine when these graphs are undirected and their regularity degrees (Proposition~\ref{prop Gamas}). 
We also study the structure of the lattice (poset) of the subgraphs for $q^m$ fixed (Proposition~\ref{teo subg}).

Section \ref{sec:3} is devoted to the computation of the spectra of these graphs. 
It turns out that they depend on some exponential sums associated to quadratic forms related to the powers $x^{q^\ell+1}$. 
Thus, we first recall quadratic forms $Q(x)$ of $m$ variables over finite fields $\ff_q$ 
and their invariants.
Then, in Lemma \ref{LsumT} we compute the exponential sums $T_{Q,a}$ in \eqref{TQ}. 
For quadratic forms 
$$Q_{\gamma, \ell}(x) = \Tr_{q^m/q}(\gamma x^{q^\ell+1}), \qquad  \gamma \in \ff_{q^m}^*,$$ 
we recall the known distribution of ranks and types due to Klapper (see Theorems \ref{Thmpar} and \ref{Thmimpar}).  
Using all these facts, in Theorem \ref{Spec Gamma} we compute the spectrum $Spec(\Gamma_{q,m}(\ell))$ of the graphs, obtaining explicit expressions for the eigenvalues and their multiplicities. 
For $\ell\ne \frac m2$ there are 3 distinct eigenvalues, while there are only 2 in the case $\ell =\frac m2$. In both cases the eigenvalues are integers.

Knowing the spectrum explicitly allows us to give some structural properties of the graphs. This is done in the next section, where we also consider the family $\bar{\mathcal{G}}_{q,m}$ of complementary graphs. 
In Proposition \ref{Spec complement} we compute the spectra of the complementary graphs, which are of course also integral.
Next, in Proposition \ref{properties}, we show that all the graphs considered are non-bipartite 
(except for $\Gamma_{2,2}(1) = 2K_2$, $\bar \Gamma_{2,2}(1) = C_4 = K_{2,2})$), mutually 
non-isospectral and connected (except when $\ell=\frac m2$, in which case it is a sum of copies of a complete graph, i.e.\@ 
$\Gamma_{q,m}(\frac m2) = q^{\frac m2} K_{q^{\frac m2}}$). 
Furthermore, we compute the number of closed walks of length $r$ for any $r$ 
(Corollary \ref{coro walks}).

All the graphs considered turn out to be strongly regular, this is the topic of Section~{\ref{sec:5}. 
For the  graphs $\Gamma_{q,m}(\ell)$ and $\bar \Gamma_{q,m}(\ell)$ we give the parameters $srg(v,k,e,d)$ as strongly regular graphs (Theorem \ref{prop srg}) as well as their intersection arrays as distance regular graphs (Corollary~\ref{coro array}).
In Proposition \ref{teo latin} we show that the connected graphs $\Gamma_{q,m}(\ell)$ in $\mathcal{G}_{q,m}$ (i.e.\@ for $\ell \ne \tfrac m2$) are of Latin square type; that is, pseudo Latin square graphs if $\frac 12 m_\ell$ is odd and negative Latin square graphs if $\frac 12 m_\ell$ is even.
Being strongly regular graphs they have diameter $\delta=2$ and girth $g=3$. We also study some other classic invariants. 
In Proposition \ref{prop invariants} we compute or give bounds for the chromatic, independence and clique numbers $\chi(\Gamma)$, $\alpha(\Gamma)$, $w(\Gamma)$, the isoperimetric constant $h(\Gamma)$ and the vertex and edge connectivities $\kappa(\Gamma)$ and 
$\lambda_2(\Gamma)$.

In Section \ref{sec:6} we study the energy of these graphs. In Proposition \ref{prop energy} we give the energies of $\Gamma_{q,m}(\ell)$ and $\bar \Gamma_{q,m}(\ell)$ for any $q,m,\ell$. In Proposition \ref{coro pairs} we exhibit four infinite families of pairs of equienergetic non-isospectral graphs involving some graphs in $\mathcal{G} \cup \bar{\mathcal{G}}$. Further, we show that all the graphs considered are hyperenergetic (with $\ell \ne \frac m2$) and do not have maximal or $k$-maximal energy with the unique exceptions of the complement of the Clebsch graph $\bar \Gamma_{2,4}(1)$ and the Clebsch graph $\Gamma_{2,4}(1)$, respectively (see Corollaries \ref{coro hyper} and \ref{comp clebsch}). However, in Proposition \ref{min Energy} we prove that the graphs are asymptotically $k$-maximal energetic.

%

In Section \ref{sec:8}, as a quite unexpected consequence, we obtain the answer to  Waring's problem for finite fields in some open cases. More precisely, we prove that the Waring number (see \eqref{WN}) for powers of the form $q^\ell+1$ is exactly 
$$g(q^{\ell}+1,q^m)=2$$ 
(or else  $g(q^{\ell}+1,2^m)=1$ if $m_\ell$ is odd), meaning that every element in $\ff_{q^m}$ can be written as a sum of at most two $(q^{\ell}+1)$-th powers. That is, for $a\in \ff_{q^m}$ there exists $x, y \in \ff_{q^m}$ such that $a=x^{q^\ell+1}+y^{q^\ell+1}$.

In the last section, we focus on the property of being Ramanujan. In Theorem \ref{Ram q23} we show that Ramanujan graphs in the family 
$\mathcal{G}_{q,m}$ can occur only in characteristics 2 and 3. Moreover, we construct 3 infinite families of Ramanujan graphs (with different degrees of regularity) over $\ff_2$, $\ff_3$ and $\ff_4$. 
In Tables \ref{tabla f2}--\ref{tabla f4} we give the parameters and the spectrum of the first three graphs in each of the families. 
The first graphs in each family are the Clebsch graph $\Gamma_{2,4}(1)=srg(16,5,0,2)$, the Brouwer-Haemers graph 
$\Gamma_{3,4}(1) = srg(81,20,1,6)$ and $\Gamma_{4,4}(1) = srg(256,51,2,12)$. 
On the other hand, every graph in $\bar{\mathcal{G}}_{q,m}$ is Ramanujan (also Theorem \ref{Ram q23}). 
Since a graph is Ramanujan if and only if the Ihara zeta function $\zeta_\Gamma(t)$ of $\Gamma$ satisfies the Riemann hypothesis in this context (\cite{Ih})}, in Proposition~\ref{ihara} we give an expression of $Z_\Gamma(t)$ for the graphs $\Gamma \in \mathcal{G}_{q,m} \cup \bar{\mathcal{G}}_{q,m}$. 
Finally, in Example \ref{ex zetas} we compute $\zeta_\Gamma(t)$ for the Clebsch and the Brouwer-Haemers graphs (and their complements) and also their complexities (the number of spanning trees).

\section{The generalized Paley graphs of $(q^\ell+1)$-th powers} \label{sec:2}
Given a group $G$ and a subset $S \subset G$ (called the connection set) the \textit{Cayley graph} $\Gamma=X(G,S)$ is the digraph with vertex set $G$ where two elements $g,h \in G$ form a directed edge or arc from $g$ to $h$ if $g^{-1}h \in S$ ($h-g \in S$ if $G$ is abelian), hence without multiple edges. If the identity $e_G \notin S$ then $\Gamma$ has no loops (hence, it is simple). If $S$ is symmetric, that is $S=S^{-1}$ (or $S=-S$ for $G$ abelian), then $\Gamma$ has undirected edges. In this case, $\Gamma$ is $k$-regular with $k=|S|$. Actually, $\Gamma$ is vertex-transitive. 
Finally, $\Gamma$ is connected if and only if $S$ is a generating set of $G$. 

For instance, the cycle and complete graphs in $n$-vertices are Cayley graphs: $C_n=X(\Z_n,\{1\})$ and 
$K_n = X(\Z_n,\Z_n \smallsetminus\{0\})$.
Another interesting examples using finite fields as vertex sets are given by the Paley graphs 
$P(q)=X(\ff_{q},S)$ with $S=\{x^2 : x \in \ff_{q}^*\}$ (notice that $P(q)$ is undirected if $q\equiv 1 \pmod 4$ and directed if $q \equiv 3 \pmod 4$), and by the generalized Paley graphs $X(\ff_q, R_k)$ with $\{x^k:x\in \ff_q^*\}$ (see \cite{PV}, \cite{PV2}, \cite{PV Equi3} for properties of these graphs in general).

\subsection{The graphs $\Gamma_{q,m}(\ell)$}
Let $q=p^s$ be a prime power with $s\ge 1$, $m \in \N$ and $\ell \in \N_0$. We will next consider a family of Cayley graphs $\Gamma_{q,m}(\ell)$
which will turn out to be a generalization of the graphs $K_{q^m}$ and $P(q^m)$.
Consider the set of $(q^\ell+1)$-th powers in $\ff_{q^m}^*$
\begin{equation} \label{Sqml}
S_\ell := S_{q,m}(\ell) = \{ x^{q^\ell+1} : x \in \ff_{q^m}^* \}.
\end{equation}

\begin{defi} \label{def Sl}
Considering $\ff_{q^m}$ additively, we define the Cayley graph 
\begin{equation} \label{gamaml}
\Gamma_{q,m}(\ell) = X(\ff_{q^m}, S_\ell). 
\end{equation}
We call it the \textit{generalized Paley graph of $(q^\ell+1)$-th powers}. 
Notice that if $\ell=0$, we obtain the Paley graph $P(q^m)$. 
We will denote the set of edges of $\Gamma_{q,m}(\ell) $ by $E_{q,m}(\ell) $, or simply by $E_\ell$. 
It is enough to consider $0\le \ell\le m-1$
since $x^{q^m-1}=1$ implies that $S_{km+\ell} = S_{\ell}$ for every $k$. 
\end{defi}

Recall that if $m\in \Z$, the $2$-adic valuation of $m$, denoted by $v_2(m)$, is the power of $2$ 
in the prime factorization of $m$.
We will need the following result, whose proof is elementary. 
\begin{lem} \label{mcd}
Let $m$ and $\ell$ be integers and $q$ a prime power. Then,
$$(q^m-1, q^\ell+1) = \left\{ \begin{array}{cl} 
q^{(m,\ell)}+1, & \qquad \text{if }  v_{2}(m)> v_2(\ell),  \\[1.25mm]
\hfil 2, & \qquad \text{if $v_{2}(m) \leq v_{2}(\ell),$ with $q$ odd},  \\[1.25mm]
\hfil 1, & \qquad \text{if $v_{2}(m) \leq v_{2}(\ell),$ with $q$ even}.
\end{array} \right.$$ 
\end{lem}

We now study the basic properties of $S_\ell$.  
For integers $m,\ell$, we will use the notation 
\begin{equation*} \label{ml}
	m_\ell = \tfrac{m}{(m,\ell)}.
\end{equation*}
We begin by showing that $S_\ell=S_{(m,\ell)}$ for $m_\ell$ even.

\begin{lem} \label{Sls}
Let $m,\ell$ be non-negative integers.
The subset $S_\ell$, defined in \eqref{Sqml}, is a multiplicative subgroup of $\ff_{q^m}^*$ and we have:
\begin{enumerate}[$(a)$] \setlength\itemsep{1mm}
\item If $q$ is even, then $S_\ell = \ff_{q^m}^*$ if $m_\ell$ is odd or $S_\ell = S_{(m,\ell)}$ if $m_\ell$ is even. 

\item If $q$ is odd, then $S_\ell = S_0$ if $m_\ell$ is odd or $S_\ell = S_{(m,\ell)}$ if $m_\ell$ is even. 
\end{enumerate}
Moreover, if $m_\ell$ is odd we have that $|S_0| = \frac{q^m-1}2$ for $q$ odd and $|S_\ell| = q^m-1$, $\ell\ne 0$, for $q$ even, while $|S_{(m,\ell)}| = \frac{q^m-1}{q^{(m,\ell)}+1}$ if $m_\ell$ is even. 
\end{lem}

\begin{proof}
$S_\ell$ is a multiplicative subgroup of $\ff_{q^m}^*$ since $x^{q^\ell+1} y^{q^\ell+1} = (xy)^{q^\ell+1}$. 
Now, if $\alpha$ is a primitive element of $\mathbb{F}_{q^m}$ then $S_\ell = \langle \alpha ^{q^{\ell}+1} \rangle$. Thus, $S_\ell$ 
is a cyclic subgroup of $\mathbb{F}_{q^m}^*$ of order 
$t=\frac{q^m-1}{(q^m-1, \, q^{\ell}+1)}$. 
Since the subgroup generated by $\alpha^{(q^m-1, \, q^{\ell}+1)}$ has the same order than $S_{\ell}$, and every cyclic group of order $n$ has a unique subgroup of order $d$ for each $d \mid n$, we get that 
	$$S_\ell = \langle \alpha^{(q^m-1, \, q^{\ell}+1)} \rangle.$$

Now, the condition $v_2(m) > v_2(\ell)$ implies that $m$ and $m_\ell$ are even. Conversely, if $m_\ell$ is even then $m$ is even and $v_2(m)>v_2(\ell)$. 
Thus, by Lemma \ref{mcd} we have that $(q^m-1, q^{\ell}+1) = q^{(m,\ell)}+1$ if $m_{\ell}$ is even for any $q$ 
and $(q^m-1, q^{\ell}+1) = 2$ (resp.\@ $1$) if $m_{\ell}$ odd for $q$ odd (resp.\@ even). 
Hence, $S_{\ell}=\{x^{q^{(m,\ell)}+1}: x\in\mathbb{F}_{q^m}^*\}$
for $m_{\ell}$ even and $S_{\ell}=\{x: x\in\mathbb{F}_{q^m}^*\}$ or $S_{\ell}=\{x^2 : x\in\mathbb{F}_{q^m}^*\}=S_0$  
for $m_{\ell}$ odd, depending on whether $q$ is even or odd respectively, and thus the result follows. 
\end{proof}

We now give conditions for the symmetry of $S_\ell$.
\begin{lem} \label{simetrico}
Let $S_\ell=S_{q,m}(\ell)$ 
be as in \eqref{Sqml}. Then, $S_\ell$ is symmetric if and only if $-1\in S_\ell$.
Moreover, we have:
\begin{enumerate}[$(a)$] \setlength\itemsep{1mm}
\item If $q$ is even, then $S_\ell$ is symmetric (for every $m$ and $\ell$). 

\item If $q$ is odd, then $S_\ell$ is symmetric if $m_\ell$ is odd and $q^m\equiv 1 \pmod 4$ or else if $m_\ell$ is even.
\end{enumerate}
In these cases, $\Gamma_{q,m}(\ell)$ is undirected.
\end{lem}

\begin{proof}
Clearly, $1\in S_\ell$, so if $S_\ell$ is symmetric then $-1\in S_\ell$.  
On the other hand, if there is some $y\in\mathbb{F}_{q^m}^*$ satisfying $y^{q^\ell+1}=-1$, then we have 
$-x^{q^\ell+1} = x^{q^\ell+1}y^{q^\ell+1} = (xy)^{q^\ell+1} \in S_\ell$ for any $x\in\mathbb{F}_{q^m}^*$,
and thus $S_{\ell}$ is symmetric.

In even characteristic $S_{\ell}=-S_{\ell}$, since $-x=x$ for every $x\in \ff_{q^m}^*$, showing (\textit{a}). 
To prove (\textit{b}), 
let $q$ be odd and $m\ge 1$, $\ell\ge 0$ integers. 
If $m_\ell$ is odd, then $S_\ell = S_0 = \{x^2:x\in \ff_{q^m}^* \}$ by the previous lemma.	In this case, $-1$ is a square in $\ff_{q^m}$ if and only if $q^m\equiv 1 \pmod 4$. 
Now, suppose that $m_\ell$ is even. As in the proof of Lemma \ref{Sls} we have 
$S_{\ell} = \langle \alpha^{q^{(m,\ell)}+1}\rangle$
for any primitive element $\alpha$ of $\mathbb{F}_{q^m}$. 
Therefore, it suffices to prove that there is some $y\in\mathbb{F}_{q^m}^*$ such that $y^{q^{(m,\ell)}+1}+1=0$.
We claim that $2(q^{(m,\ell)}+1)\mid q^m-1$.
In fact, we have that $q^m-1 = (q^{m-1}+\cdots+q^2+q+1)(q-1)$ and since $q^{(m,\ell)}+1$ and $q-1$ are coprime then 
	$$q^{(m,\ell)}+1 \mid q^{m-1}+\cdots+q^2+q+1.$$ 
Also, since $q$ is odd we have $2\mid q-1$. Thus, $2(q^{(m,\ell)}+1)\mid q^m-1$. 
Hence, there exists some positive integer $t$ such that 
$q^m-1=2t(q^{(m,\ell)}+1)$.
Let $y=\alpha^t$. Note that $y^{2(q^{(m,\ell)}+1)}=1$, and hence 
$y^{q^{(m,\ell)}+1}=\pm 1$. 
Since $\alpha$ is a primitive element, the order of  $y$ is $2(q^{(m,\ell)}+1)$. Thus, we have  
$y^{q^{(m,\ell)}+1}=-1$.
In this way, $S_{\ell}=-S_{\ell}$ as we wanted to see.
\end{proof}

Summing up, we have the following description of the graphs considered.
\begin{prop} \label{prop Gamas}
For $q$ a prime power and $0\le \ell \le m-1$ we have  
\begin{equation} \label{Gamas}
\Gamma_{q,m}(\ell) = \begin{cases}
K_{q^m}, 					  	 & \qquad \text{if $m_\ell$ odd, $q$ even,} \\[.5mm]
\Gamma_{q,m}(0)=P(q^m), 		 & \qquad \text{if $m_\ell$ odd, $q$ odd,} \\[.5mm]
\Gamma_{q,m}((m,\ell)),			 & \qquad \text{if $m_\ell$ even (any $q$).} 
\end{cases}
\end{equation}
Thus, $\Gamma_{q,m}(\ell)$ is undirected, except in the case $\ell=0$ and $q\equiv 3 \pmod 4$. Also, $\Gamma_{q,m}(\ell)$ is $k$-regular, with $k=q^m-1, \frac{q^m-1}2$ or $\frac{q^m-1}{q^\ell+1}$, respectively.
\end{prop}

\begin{proof}
By Lemma \ref{Sls}, if $m_\ell$ is odd, then $\Gamma_{q,m}(\ell)$ is the complete graph $K_{q^m}$ when $q$ is even and it is the Paley graph $P(q^m)$ when $q$ is odd. By Lemma \ref{simetrico}, $\Gamma_{q,m}(\ell)$ is an undirected graph, except in the case 
$q \equiv 3 \pmod 4$ with $m_\ell$ odd. The remaining assertion follows from Lemma \ref{Sls} again. 
\end{proof}

By \eqref{Gamas}, in the more interesting case when $m_\ell$ is even, it is enough to consider the graphs 
for positive divisors of $m$. 
Thus, we consider the family
\begin{equation} \label{family}
\mathcal{G}_{q,m} = \{ \Gamma_{q,m}(\ell) : 1\le \ell \le [\tfrac m2] \text{ with $\ell \mid m$ and $\tfrac{m}{\ell}$ even}\}.
\end{equation}
All the graphs in $\mathcal{G}_{q,m}$ are not isomorphic to each other because they have different degrees of regularity; 
in fact, $\Gamma_{q,m}(d)$ is $\frac{q^m-1}{q^d+1}$-regular, by Lemma \ref{Sls}.

\begin{rem} \label{GP}
The graphs $\Gamma_{q,m}(\ell)$ with $\ell \mid m$ are, in particular, generalized Paley graphs $GP(q^m,\frac{q^m-1}{k})$ as defined in \cite{LP}, with 
$k=q^\ell+1$. By Theorems 1.2 and 2.2 in \cite{LP}, the graph 
$\Gamma_{q,m}(\ell) = GP(q^m,\tfrac{q^m-1}{q^\ell+1})$ 
is connected if and only $q^\ell+1$ is not a multiple of $\frac{q^m-1}{p^d-1}$ for every proper divisor $d$ of $rm$, where $q=p^r$. 
This fact, together with the factorization $p^{rm}+1 = (p^{\frac{rm}{2}}-1)(p^{\frac{rm}{2}}+1)$, implies that $\Gamma_{q,m}(\frac m2)$ is not connected. However, it is not clear what happens for $\ell \ne \frac m2$ in general. 
This will be answered in Proposition \ref{properties}.
\end{rem}

\subsection{Subgraphs and lattices}
We now study conditions on the parameters $q_i,m_i,\ell_i$, $i=1,2$, such that $\Gamma_{q_1,m_1}(\ell_1)$ be a subgraph of 
$\Gamma_{q_2,m_2}(\ell_2)$, i.e.\@ 
$\ff_{q_1^{m_1}} \subset \ff_{q_2^{m_2}}$ and $E_{q_1,m_1}(\ell_1) \subset E_{q_2,m_2}(\ell_2)$.
There are two trivial cases using the same field and different powers or using the same power in different fields. Namely, 
$$ (i) \:\: S_L \subset S_\ell \:\:\Rightarrow \:\: \Gamma_{q,m}(L) \subset \Gamma_{q,m}(\ell), \qquad \quad (ii) \:\:
\ff_{q^m} \subset \ff_{q^M} \:\: \Rightarrow \:\: \Gamma_{q,m}(\ell) \subset \Gamma_{q,M}(\ell).$$ 
The case in ($ii$) holds if and only if $m \mid M$. In this case, if $m<M$, the subgraph has lower order than the graph. We thus study the case in ($i$), in which the graph and subgraph have the same number of vertices.
In the previous notations we have the following.
\begin{lem} \label{SLcSl}
Let $q, \ell, L, m \in \N$ with $q$ a prime power. 
If $\ell, L$ are divisors of $m$, with $m_\ell$ and $m_L$ even, then
$$S_L \subset S_\ell \qquad \Leftrightarrow \qquad q^\ell +1 \mid q^L+1 
\qquad \Leftrightarrow \qquad \ell \mid L \quad \text{with} \quad \tfrac{L}{\ell} \quad \text{odd.}$$
\end{lem}

\begin{proof}
For the first equivalence, since $S_L$ is a subgroup of $S_\ell$, the order $|S_L|$ divides the order $|S_\ell|$, 
that is $\frac{q^m-1}{(q^m-1,q^L+1)}$ divides $\frac{q^m-1}{(q^m-1,q^\ell+1)}$. Since $m_\ell, m_L$ are even and both $\ell$ and  $L$ divides 
$m$, we have $(q^m-1,q^\ell+1)=q^\ell+1$ and $(q^m-1,q^L+1)=q^L+1$, by Lemma \ref{mcd}, and hence $q^L+1 = (q^\ell+1)k$ for some $k$.
Conversely, $q^L+1 = (q^\ell+1)k$ for some $k$ clearly implies that $x^{q^L+1}=(x^k)^{q^\ell+1} \in S_\ell$ and thus $S_L \subset S_\ell$.

For the second equivalence, if $q^\ell+1 \mid q^L+1$, there exist some $k \in \N$ such that  
$$ q^L+1 = (q^\ell+1)(q^{L-\ell} - q^{L-2\ell} + q^{L-3\ell} - \cdots - q^{L-(k-1)\ell} + q^{L-k\ell})$$ 
with $q^{L-k\ell}=1$, that is $L=k\ell$. The alternation of signs implies that $k$ is odd.
Conversely, let $L=k\ell$ with $k$ odd, then 
$q^k+1 = (q+1)(q^{k-1}-q^{k-2} + \cdots -q + 1)$.
By changing $q$ by $q^\ell$ above we get 
$q^L+1 = (q^\ell+1)(q^{\ell(k-1)}-q^{\ell(k-2)} + \cdots -q^\ell+1)$, as desired. 
\end{proof}

We now give necessary and sufficient conditions for $\Gamma_{q, m_1}(\ell_1)$ to be a subgraph of $\Gamma_{q, m_2}(\ell_2)$.
\begin{prop} \label{teo subg}
Let $q$ be a prime power, and for $i=1,2$ let $p_i$ be a prime and $r_i, \ell_i, m_i$ be positive integers with 
$\ell_i \mid m_i$ and $m_{\ell_i}$ even. Then, we have:
\begin{enumerate}[$(a)$] \setlength\itemsep{1mm}
\item $\Gamma_{q, m_1}(\ell_1) \subset \Gamma_{q,m_2}(\ell_2)$ if and only if 
$m_1 \mid m_2$ and $\ell_2 \mid \ell_1$ with $\frac{\ell_1}{\ell_2}$ odd. 

\item $\Gamma_{p_1^{r_1}, m_1}(\ell_1) = \Gamma_{p_2^{r_2},m_2}(\ell_2)$ if and only if 
$p_1=p_2$, $r_1 m_1 = r_2 m_2$ and $r_1 \ell_1 = r_2 \ell_2$.
\end{enumerate}
\end{prop}

\begin{proof}
(\textit{a}) Let $\Gamma_i=\Gamma_{q, m_i}(\ell_i)$, $i=1,2$. Now, $\Gamma_1 \subset \Gamma_2$ if and only if 
$V_1=\ff_{q^{m_1}} \subset V_2=\ff_{q^{m_2}}$ and 
$E_{q,m_2}(\ell_1) \subset E_{q,m_2}(\ell_2)$. The first contention is equivalent to $m_1\mid m_2$ and the second one to 
$S_{q,m_1}(\ell_1) \subset S_{q,m_2}(\ell_2)$, by considering the edges in $E_{\ell_1}$ containing the vertex $0$. The result thus follows by Lemma \ref{SLcSl}.

\noindent (\textit{b}) 
$\Gamma_1 = \Gamma_2$ if and only if $V_1=V_2$ and $E_1=E_2$.
The equality $V_1=V_2$ of the vertex sets is equivalent to $p_1^{r_1m_1} = p_2^{r_2m_2}$, hence $p_1=p_2=p$ and $r_1m_1 = r_2m_2$. 
The equality $S_1=S_2$ between the connection sets is $\{x^{p^{r_1\ell_1}+1} : x\in \ff_{q}^*\} = 
\{ x^{p^{r_2\ell_2}+1} : x\in \ff_{q}^* \}$, with $q=p^{r_1m_1}=p^{r_2m_2}$. 
By Lemma \ref{SLcSl}, this happens if and only if $r_1m_1=r_2m_2$ and $r_1\ell_1 = r_2\ell_2$, since 
$\frac{r_2\ell_2}{r_1\ell_1}=1$ is odd.
\end{proof}

\begin{exam}
(\textit{i}) 
We have $\mathcal{G}_{q,12}=\{ \Gamma_{q,12}(1), \Gamma_{q,12}(2), \Gamma_{q,12}(3), \Gamma_{q,12}(6)\}$, 
$\Gamma_{q,12}(4)$ is excluded because $m_4=3$ is odd. We have $\Gamma_{q,12}(6) \subset \Gamma_{q,12}(2)$ and 
$\Gamma_{q,12}(3) \subset \Gamma_{q,12}(1)$.

\noindent 
(\textit{ii}) 
The graphs in $\mathcal{G}_{q,30}$ are $\Gamma_{q,30}(1)$, $\Gamma_{q,30}(3)$, $\Gamma_{q,30}(5)$ and 
$\Gamma_{q,30}(15)$ where we have ruled out 
$\Gamma_{q,30}(10)$, $\Gamma_{q,30}(6)$ and $\Gamma_{q,30}(2)$. 
We have $\Gamma_{q,30}(15) \subset \Gamma_{q,30}(3)$, $\Gamma_{q,30}(5) \subset \Gamma_{q,30}(1)$. 

\noindent 
(\textit{iii}) 
If $m=2p^t$, with $p$ odd and $t\ge 1$, then we have the chain of proper subgraphs
$$\Gamma_{q,m}(p^{t}) \subset  \Gamma_{q,m}(p^{t-1}) \subset \cdots \subset \Gamma_{q,m}(p^{2}) \subset \Gamma_{q,m}(p) \subset 
\Gamma_{q,m}(1).$$
All the graphs are connected, except for the first one since $\ell=\tfrac m2$ (see Proposition \ref{properties} ahead).
\end{exam}

\begin{rem} \label{rem39}
If $p,r,\ell,m$ are non-negative integers, with $p$ prime and $\ell \mid m$, then
\begin{equation} \label{mascaradas} 
\Gamma_{p, rm}(r\ell) = \Gamma_{p^r, m}(\ell) = \Gamma_{p^{r\ell}, \frac{m}{\ell}}(1).
\end{equation}
The first equality says 
that we can restrict ourselves to the case of $m,\ell$ coprime by taking $d=(m,\ell)$, i.e.\@
$\Gamma_{p, m}(\ell) = \Gamma_{p^d, \frac md}(\tfrac{\ell}d)$. 
The second equality implies that we can always take the simplest connection set $S_1 = \{x^{q+1}:x\in \ff_q\}$.
\end{rem}

The lattice (poset) of divisors of $m$, that we denote by $\Lambda(m)$, together with Proposition \ref{teo subg} induce a lattice structure on a subset of $\mathcal{G}_{q,m}$ in \eqref{family} which, since $q$ is fixed, we will denote by 
$\Lambda(\mathcal{G}(m))$.
We will see that this structure is rather simple. 
  
\begin{prop}
Let $m=2^t r$ with $r$ odd. The lattice $\Lambda(\mathcal{G}(m))$ is empty if $m$ is odd while it is isomorphic to $t$ copies of the lattice of divisors of $r$ if $m$ is even. More precisely, $\Lambda(\mathcal{G}(r)) =\varnothing$ for $t=0$ and, 
if $\sqcup$ denotes disjoint union, for $t\ge 1$ we have 
\begin{equation} \label{lattice}
\Lambda(\mathcal{G}(2^t r)) = \Lambda(\mathcal{G}(r)) \sqcup \Lambda(\mathcal{G}(2r)) \sqcup \cdots 
\sqcup \Lambda(\mathcal{G}(2^{t-1} r)). 
\end{equation}
Also, $\Lambda(\mathcal{G}(2^k r)) \simeq 2^k \Lambda(r)$ for each $1\le k \le t$ where 
$a\Lambda(r) = \{ad : d \in \Lambda(r)\}$.
Thus $\Lambda(\mathcal{G}(2^t r))$ has $t$ connected components and hence it is connected if and only if $t=\nu_2(m)=1$.
\end{prop}

\begin{proof}
We will use the notations 
\begin{equation} \label{Leo}
	\Lambda_e(m) = \{d\mid m : \tfrac md \text{ is even} \} \qquad \text{and} \qquad 
	\Lambda_o(m) = \{d\mid m : \tfrac md \text{ is odd}\}. 
\end{equation}
Note that $\Lambda_e(m) = \Lambda(\tfrac m2)$ if $m$ is even and $\Lambda_e=\varnothing$ if $m$ is odd. Also, 
$\Lambda_o(m) = \Lambda(m)$ if $m$ is odd and $\Lambda_o = \varnothing$ if $m$ is even. 
Thus, by \eqref{family}, \eqref{Leo}, and the previous observations we have
	$$\Lambda(\mathcal{G}(m)) \simeq (\Lambda_e)_o (m) = \Lambda_o (\tfrac m2)$$ 
where the isomorphism sends $\Gamma_{q,m}(d)$ to $d$ for $d\mid m$.
Now, it is clear that $\Lambda(\mathcal{G}(m))$ is empty if $t=0$ ($m$ odd). If $t\ge 1$ we have 
$\Lambda(\mathcal{G}(2^t r)) \simeq \Lambda _o(2^{t-1} r)$.
If $r=p$ is a prime, the result in \eqref{lattice} is clear, 
and hence 
	$$\Lambda(2^{t-1}r) = \Lambda(r) \sqcup 2\Lambda(r) \sqcup \cdots \sqcup 2^{t-1} \Lambda(r).$$ 
The case of $m=2^t p^k$ is a slightly more involved.
The general result in \eqref{lattice} follows by induction on the number of primes of $r$. We leave the details to the reader. 
\end{proof}

\section{The spectra of generalized Paley graphs} \label{sec:3}
It is well known that the spectrum of a Cayley graph $\Gamma=X(G,S)$, with $G$ an abelian group, is determined by the 
characters of $G$. More precisely, each character $\chi$ of $G$ induces an eigenvalue of $\Gamma=X(G,S)$ by the expression 
\begin{equation} \label{Eigencayley}
\chi(S) = \sum_{g\in S} \chi(g).
\end{equation} 
When we consider the graphs $\Gamma_{q,m}(\ell)$, the sums in \eqref{Eigencayley} will be exponential sums 
associated to certain Gauss sums or to quadratic forms determined by the trace function and the exponent $q^\ell+1$. 
We have chosen the quadratic form approach.

\subsection{Quadratic forms and exponential sums} 
A quadratic form in $\ff_{q^m}$ is an homogeneous polynomial of degree 2 in $\ff_{q^m}[x]$. 
More generally, any function 
$Q : \mathbb{F}_{q^m}\rightarrow \mathbb{F}_{q}$
can be identified with a polynomial of $m$ variables over $\mathbb{F}_{q}$ via an isomorphism $\mathbb{F}_{q^m}\simeq\mathbb{F}_{q}^m$ of $\mathbb{F}_{q}$-vector spaces. 
Such $Q$ is said to be a \textit{quadratic form in $m$ variables} if the corresponding polynomial is homogeneous of degree $2$. 
The minimum number of variables needed to represent $Q$ as a polynomial in several variables is a well defined number $r$ 
called the \textit{rank} of $Q$.
Two quadratic forms $Q_{1},Q_{2}$ are \textit{equivalent} if there is an invertible $\mathbb{F}_{q}$-linear function $S:\mathbb{F}_{q^m}\rightarrow\mathbb{F}_{q^m}$ such that $Q_1(x) =  Q_{2}(S(x))$.

Fix $Q$ a quadratic form from $\mathbb{F}_{q^m}$ to $\mathbb{F}_{q}$ and $\xi \in \mathbb{F}_{q}$. Consider the number  
\begin{equation} \label{NQ}
N_{Q}(\xi) = \# \{ x\in \ff_{q^m} : Q(x) = \xi \}.
\end{equation}
We will abbreviate 
$N_Q = N_Q(0)=\#\ker Q$.
It is a classical result that there are three equivalence classes of quadratic forms over finite fields. 
This classification depends on the parity of the characteristic (see for instance \cite{LN} or \cite{MP}). 
In both characteristics, there are 2 classes with even rank and one with odd rank. In the case of even rank, it is usual to say that 
$Q$ is of type I or III, but it will be convenient for us to call them of type $1$ or $-1$, respectively. 
Thus, we will put 
$$\varepsilon_Q = \pm 1 \qquad \quad \text{if $Q$ is of type $\pm 1$}$$ 
and we will call this sign $\varepsilon_Q$ the \textit{type} of $Q$.
For even rank, the number 
$N_Q(\xi)$ does not depend on the characteristic and it is given as follows
\begin{equation} \label{NQ+-} 
N_Q(\xi) = q^{m-1} + \varepsilon_Q  \, \nu(\xi) \, q^{m-\frac{r}{2}-1}
\end{equation}
where 
$\nu(0)=q-1$ and $\nu(z)=-1$ if $z\in \mathbb{F}_{q}^*$.

Recall that in general, the trace map from $\ff_{q^m}$ to $\ff_q$, with $q$ a prime power, is defined by
$$\Tr_{q^m/q}(x) = x^{q^{m-1}} + \cdots + x^{q^2} + x^q + x$$ 
for $x\in \ff_{q^m}$. 
For $Q$ as before, $a \in \mathbb{F}_{q}^*$ and $\zeta_p=e^{\frac{2\pi i}{p}}$ we will need the following exponential sum
\begin{equation}\label{TQ}
T_{Q,a} = \sum_{x\in\mathbb{F}_{q^m}}{\zeta_{p}^{\Tr_{q/p}(aQ(x))}}. 
\end{equation}
We will abbreviate $T_{Q,1}=T_Q$.
Notice that $T_{Q,a} \in \mathbb{C}$. We next give the values of $T_{Q,a}$ in the case of even rank.

\begin{lem} \label{LsumT}
If $Q$ be a quadratic form of $m$ variables over $\mathbb{F}_{q}$ of even rank $r$.
Then, 
for all $a\in \mathbb{F}_{q}^*$ we have $T_{Q,a} = T_Q = \varepsilon_{Q} \, q^{m-\frac{r}{2}}$.
\end{lem}

\begin{proof} 
Since $r$ is even, by \eqref{NQ+-} we have  
$N_Q = q^{m-1} + \varepsilon_{Q} (q-1) q^{m-\frac{r}{2}-1}$.
On the other hand, note that for any $a\in \ff_q$ fixed, 
$\chi_{a}(y) = \zeta_p^{\Tr_{q/p}(ay)}$ with $y\in \ff_q$ and $\zeta_{p}=e^{\frac{2\pi i}{p}}$
is a character of $\ff_{q}$. 
Thus, by orthogonality of the characters we get, using \eqref{TQ}, that 
\begin{equation*} \label{NQ2}
N_Q = \tfrac{1}{q} \sum_{x\in\ff_{q^m}} \sum_{a\in\ff_q} \chi_a(Q(x)) =
\tfrac{1}{q} \sum_{x\in\ff_{q^m}} \sum_{a\in\ff_q} \zeta_{p}^{\Tr_{q/p}(aQ(x))} =
q^{m-1} + \tfrac{1}{q} \sum_{a\in\ff_q^*} T_{Q,a}. 
\end{equation*}
Therefore, by equating the two previous expressions for $N_Q$ we arrive at 
\begin{equation} \label{TQ2}
 \sum_{a\in\ff_q^*} T_{Q,a} = \varepsilon_{Q} (q-1) q^{m-\frac{r}{2}}.
\end{equation}

It is known that $|T_Q| = q^{m-\frac r2}$ or $|T_Q|=0$ (see \cite{FL} for $q$ odd, \cite{LTW} for $q$ even).
Notice that $Q_{a}(x):=aQ(x)$, with $a\in \ff_{q}^*$, is a quadratic form of the same rank of $Q(x)$, that is $r_{Q_a}=r_Q=r$. 
Thus, 
$$T_{Q,a} = \sum_{x\in\mathbb{F}_{q^m}}{\zeta_{p}^{\Tr_{q/p}(aQ(x))}} = 
\sum_{x\in\mathbb{F}_{q^m}}{\zeta_{p}^{\Tr_{q/p}(Q_a(x))}} = T_{Q_a}$$
and hence $|T_{Q,a}| = q^{m-\frac r2}$ or $|T_{Q,a}| = 0$ for every $a\in \ff_q^*$.
Now, by \eqref{TQ2} and the triangle inequality, we have 
$$\sum_{a\in \ff_q^*} |T_{Q,a}|\le(q-1) q^{m-\frac r2}=|\sum_{a\in \ff_q^*} T_{Q,a}| \le \sum_{a\in \ff_q^*} |T_{Q,a}|.$$
This implies that $|T_{Q,a}|\ne 0$ and $|T_{Q,a}|=q^{m-\frac{r}{2}}$ for every $a\in \ff_{q}^*$. 
In particular, $T_Q\ne 0$. 
Thus, 
$$|\sum_{a\in\ff_q^*}T_{Q,a}|=\sum_{a\in\ff_q^*}|T_{Q,a}| =(q-1) \, q^{m-\frac r2}.$$ 
Since equality holds in the triangular inequality we have that 
$T_{Q} = T_{Q,a}$ for every $a\in\ff_q^*$. By \eqref{TQ2}, we must have that 
$T_{Q} = \varepsilon_{Q} \, q^{m-\frac{r}{2}}$, 
as we wanted. 
\end{proof}

\subsubsection*{Trace forms}
A whole family of quadratic forms over $\ff_q$ in $m$ variables, usually called `trace forms', are given by  
\begin{equation} \label{QRx}
Q_R(x) = \Tr_{q^m/q}(xR(x))
\end{equation} 
where $R(x)$ is a $q$-linearized polynomial over $\ff_q$ (i.e.\@ $R(x) =\sum_{i=0}^t a_i x^{q^i}$).

\begin{rem} \label{generalized}
If one considers two $q$-linearized polynomials $S(x), R(x)$ in \eqref{QRx}, the function
$Q_{S,R}(x) = \Tr_{q^m/q}(S(x)R(x))$ also gives a quadratic form (notice that if $S(x)=x$ then $Q_{S,R}(x)=Q_R(x)$). 
If we take monomials, 
say $S(x)= \gamma_1 x^{q^{\ell}+1}$ and $R(x) = \gamma_2 x^{q^{\ell_2}+1}$ with $\ell_1\ne\ell_2$, then 
$S(x)R(x) = \gamma x^{q^{\ell_1} + q^{\ell_2}}$.
Thus, the associated Cayley graph $\Gamma_{S,R} = X( \ff_{q^m}, \{ x^{q^{\ell_1} + q^{\ell_2}} : x \in \ff_{q^m}^* \} )$
is one of the graphs defined in \eqref{gamaml}; 
namely $\Gamma_{q,m}(\ell)$ with $\ell = \max\{\ell_1, \ell_2 \}-\min \{ \ell_1, \ell_2 \}$. 
\end{rem}

We are interested in the case when $R(x) = \gamma x^{q^\ell}$ with 
$\ell \in \N$, $\gamma \in \ff_{q^m}^*$, i.e.\@
\begin{equation} \label{Qell}
Q_{\gamma, \ell}(x) = \Tr_{q^m/q}(\gamma x^{q^\ell+1}). 
\end{equation}
The next theorems, due to Klapper, give the distribution of ranks and types of the family 
$\{ Q_{\gamma, \ell}(x) : \gamma \in \mathbb{F}_{q^m}, \ell \in \N \}$ of quadratics forms.

In the sequel, we will need the following notation
\begin{equation} \label{epsilon l}
\varepsilon_\ell = (-1)^{\frac 12 m_\ell} = \begin{cases}
\hfil 1, & \qquad \text{if {\small $\tfrac 12$}$m_\ell$ is even}, \\[1.1mm]
-1, 	 & \qquad \text{if {\small $\tfrac 12$}$m_\ell$ is odd}, 
\end{cases}
\end{equation}
where $m_\ell=\tfrac{m}{(m,\ell)}$.

\begin{thm}[even characteristic (\cite{K1})] \label{Thmpar}
Let $q$ be a power of $2$, let $m,\ell \in \N$ such that $m_\ell$ is even and let $S_\ell$ be as in \eqref{Sqml}. Then $Q_{\gamma,\ell}$ is of even rank and we have:
\begin{enumerate}[$(a)$]
\item If $\varepsilon_\ell = \pm 1$ and $\gamma \in S_\ell$ 
then $Q_{\gamma, \ell}$ is of type $\mp 1$ and has rank $m-2{(m,\ell)}$. \msk
		
\item If $\varepsilon_\ell = \pm 1$ and $\gamma \notin S_\ell$ 
		then $Q_{\gamma,\ell}$ is of type $\pm 1$ and has rank $m$.
\end{enumerate}
\end{thm}

For $q$ odd and $N \in \N$, consider the following set of integers 
\begin{equation} \label{Xqml}
Y_{q,m,\ell}(N) = \{0 \le t \le q^m-1 \, : \, t \equiv N  \, (\text{mod } L) \}
\end{equation}
where $L=q^{(m,\ell)}+1$.

\goodbreak 
\begin{thm}[odd characteristic (\cite{K2})] \label{Thmimpar}
Let $q$ be a power of an odd prime $p$ and let $m, \ell$ be non negative integers. Put $\gamma = \alpha^{t}$ with $\alpha$ a primitive element in $\ff_{q^m}$. Then, we have:
	\begin{enumerate}[$(a)$]
		\item If $\varepsilon_\ell = 1$ and $t \in Y_{q,m,\ell}(0)$ 
		then $Q_{\gamma,\ell}$ is of type $-1$ and has rank $m-2(m,\ell)$. \sk
		
		\item If $\varepsilon_\ell = 1$ and $t\notin Y_{q,m,\ell}(0)$ 
		then $Q_{\gamma,\ell}$ is of type $1$ and has rank $m$. \sk
		
		\item If $m_\ell$ is even, $\varepsilon_\ell = - 1$ and $t \in Y_{q,m,\ell}(\tfrac L2)$ 
		then $Q_{\gamma,\ell}$ is of type $1$ and has rank $m-2{(m,\ell)}$. \sk 
		
		\item If $m_\ell$ is even, $\varepsilon_\ell = -1$ and $t \notin Y_{q,m,\ell}(\tfrac L2)$ 
		then $Q_{\gamma,\ell}$ is of type $-1$ and has rank $m$.
\end{enumerate}
\end{thm}

\subsection{The spectrum of $\Gamma_{q,m}(\ell)$}
It is well-known that $Spec(K_n) = \{ [n-1]^1, [-1]^{n-1}\}$ and 
$$Spec(P(q)) = \{ [\tfrac{q-1}2]^1, [\tfrac{-1-\sqrt q}2]^{2t}, [\tfrac{-1+\sqrt q}2]^{2t}\}$$
where $q=4t+1$. Hence $K_n$ is always integral and $P(q)$ is integral for $q$ a square $q=p^{2m}$.
We now 
	compute the spectra of the graphs in the family $\mathcal{G}_{q,m}$ (thus completing the determination of the spectrum of all the graphs $\Gamma_{q,m}(\ell)$, see \eqref{Gamas}). 
\begin{thm} 
\label{Spec Gamma}
Let $q$ be a prime power and $m\ge 2, \ell > 0$ integers such that $\ell \mid m$ with $\tfrac{m}{\ell}$ even.  
{\blue Thus,} the spectrum of $\Gamma_{q,m}(\ell)$ is integral and given as follows: 
\begin{enumerate}[$(a)$]
\item If $\ell \ne \frac m2$, then the eigenvalues are 
\begin{equation} \label{Autval}
	k_\ell = \frac{q^{m}-1}{q^\ell+1}, \qquad \quad \upsilon_\ell = \frac{\varepsilon_\ell \, q^{\frac{m}{2}}-1}{q^\ell+1}, \qquad \quad 
	\mu_\ell = \frac{-\varepsilon_\ell \, q^{\frac{m}{2}+\ell}-1}{q^\ell+1},
\end{equation} 
where $\varepsilon_\ell = (-1)^{\frac 12 m_\ell}$, with corresponding multiplicities $1$, $q^\ell k_\ell$ and $k_\ell$. \sk  

\item If $\ell = \frac m2$, then the eigenvalues are 
$k_{\frac m2} = q^{\frac m2} -1$ and $\upsilon_{\frac m2} = -1$ with corresponding multiplicities $q^{\frac m2}$ and
 $q^{\frac m2}(q^{\frac m2}-1)$. 
\end{enumerate}
\end{thm}

\begin{proof}
\noindent (\textit{a}) 
We know that $\Gamma = \Gamma_{q,m}(\ell)$ is $k_\ell$-regular with $k_\ell = |S_\ell| = \frac{q^m-1}{q^\ell+1}$, by Lemma \ref{Sls}.

Now, the character group $\widehat{\mathbb{F}}_{q^m}$ of 
$(\mathbb{F}_{q^m},+)$ is cyclically generated by the canonical character $\chi(x) = \zeta_{p}^{\Tr_{q^m/p}(x)}$, where $\zeta_{p} = e^{\frac{2\pi i}{p}}$ and $p$ is the 
characteristic of $\ff_q$. 
That is, every character $\chi$ of $\mathbb{F}_{q^m}$ is of the form
\begin{equation} \label{chi gama}
\chi_\gamma(x) = \zeta_{p}^{\Tr_{q^m/p}(\gamma x)} = \chi(\gamma x), \qquad \gamma\in \mathbb{F}_{q^m}.
\end{equation} 
By \eqref{Eigencayley} the eigenvalues of $\Gamma$ are of the form
$\chi_{\gamma}(S_\ell) = \sum_{y\in S_\ell} \chi_{\gamma}(y)$.
In this way, for $\gamma \in \mathbb{F}_{q^m}$, by \eqref{chi gama} and \eqref{Sqml}, since every $x^{q^\ell+1} \in S_\ell$ is obtained
by $q^\ell+1$ different elements in $\ff_{q^m}$ we have  
\begin{equation} \label{chi gama 2}
\chi_{\gamma}(S_\ell) = \tfrac{1}{q^\ell+1} \sum_{x\in \mathbb{F}_{q^m}^*} \zeta_{p}^{\Tr_{q^m/p}(\gamma x^{q^\ell+1})} 
= \tfrac{1}{q^\ell+1} \sum_{x\in \mathbb{F}_{q^m}^*} \zeta_{p}^{\Tr_{q/p} \Tr_{q^m/q}(\gamma x^{q^\ell+1})}.
\end{equation}
That is to say, if $Q_{\gamma,\ell}(x)$ is the quadratic form defined in \eqref{Qell}, we have 
\begin{equation*} \label{gamma aux1}
\chi_{\gamma}(S_{\ell}) = \tfrac{1}{q^{\ell}+1} \sum_{x\in \mathbb{F}_{q^m}^*} \zeta_{p}^{\Tr_{q/p} (Q_{\gamma,\ell}(x))}.
\end{equation*} 
Thus, by \eqref{TQ} and taking into account the contribution of $x=0$, we have {\blue that}
\begin{equation} \label{gamma aux2}
\chi_{\gamma}(S_{\ell}) = \tfrac{T_{Q_{\gamma,\ell}}-1}{q^\ell+1}
\end{equation} 
where $T_{Q_{\gamma,\ell}}$ is the exponential sum given in  \eqref{TQ}. 

By \eqref{gamma aux2}, the eigenvalues of $\Gamma$ are given by the different values that $T_{Q_{\gamma,\ell}}$ can take. 
If $\gamma=0$, then $Q_{0,\ell}=0$ and by 
\eqref{TQ} we have that $T_Q=q^m$, obtaining the eigenvalue 
$$\chi_0(S_\ell) = \tfrac{q^m-1}{q^\ell+1} = k_\ell.$$ 

For the other eigenvalues, first note that by Theorems \ref{Thmpar} and \ref{Thmimpar}, since $m_\ell$ is even, the quadratic form
$Q_{\gamma,\ell}$ has even rank $r$ for every $\gamma$ and also $r$ takes only two possible values, 
$r \in \{ m, m-2\ell \}$.     
By Lemma \ref{LsumT}, we have that
\begin{equation} \label{TQeigen}
T_{Q_{\gamma,\ell}} = \varepsilon_{Q_{\gamma,\ell}} \, q^{m-\frac r2}.
\end{equation}
By parts ($a$) and ($b$) of Theorems \ref{Thmpar} and \ref{Thmimpar}, 
if $\varepsilon_{\ell}=(-1)^{\frac{1}{2}m_{\ell}}$ we have that 
$\varepsilon_{Q_{\gamma,\ell}}=\varepsilon_\ell $ if $r=m$ and 
$\varepsilon_{Q_{\gamma,\ell}}=-\varepsilon_\ell $ if $r=m-2\ell$.
Taking into account the two possible values of $r$ and the two signs in \eqref{TQeigen}, 
by \eqref{gamma aux2} we have that the eigenvalues of $\Gamma$ are exactly those given in \eqref{Autval}.

\sk
Next, we compute the multiplicities of the eigenvalues. First, since the eigenvalue $k_\ell$ is obtained from 
$Q_{0,\ell}(x) =0$ and we have that $T_{Q_{\gamma,\ell}} = \varepsilon_{Q_{\gamma,\ell}} \, q^{m-\frac r2} \ne q^m$ for $\gamma\ne 0$, it is clear that the multiplicity of $k_\ell$ is $1$. 
The eigenvalues 
$$\upsilon_\ell^{\pm} = \tfrac{\pm q^{\frac m2}-1}{q^\ell+1} \qquad \text{ and } \qquad  
\mu_\ell^{\pm} = \tfrac{\pm q^{\frac m2+\ell}-1}{q^\ell+1}$$ 
are obtained when $r=m$ and $r=m-2\ell$, respectively. 

If $q$ is even, 
Theorem \ref{Thmpar} give us the multiplicities of the eigenvalues; they are given by 
$$q^m-1-M, \, 0, \, 0, \, M \qquad \text{and} \qquad 0, \, q^m-1-M, \, M,\,  0$$
respectively for $\tfrac 12 m_\ell$ even and $\tfrac 12 m_\ell$ odd, where $M=\# S_{q,m}(\ell)$ with $S_{q,m}(\ell)$ as in \eqref{Sqml}.
If $\alpha$ is a primitive element of $\ff_{q^m}$, then $S_{q,m}(\ell) = \langle \alpha^{q^\ell +1} \rangle$. 
Thus, $M = \text{ord}(\alpha^{q^\ell+1}) = \tfrac{q^m-1}{(q^m-1,q^\ell+1)}$ and therefore 
$M = \frac{q^m-1}{q^\ell+1}=k_\ell$, by Lemma \ref{mcd} since $m_\ell$ is even.
 
On the other hand, if $q$ is odd, by Klapper's Theorem \ref{Thmimpar}, the corresponding multiplicities are given by
$$q^m-1-M_1, \, 0, \, 0, \, M_1 \qquad \text{ and } \qquad 0, \, q^m-1-M_2, \, M_2, \, 0$$ 
respectively for $\tfrac 12 m_\ell$ even and $\tfrac 12 m_\ell$ odd, where 
$M_1$ and $M_2$ are given by
\begin{equation} \label{T1}
M_1 =  \# Y_{q,m,\ell}(0) \qquad \text{ and } \qquad M_2 = \# Y_{q,m,\ell}(\tfrac L2),
\end{equation}
with $L=q^{(m,\ell)}+1$ and 
$Y_{q,m,\ell}(N)$ as in \eqref{Xqml}.

Let us see that $M_1=M_2=k_\ell$, and thus the multiplicities of the eigenvalues are as in the statement.
Notice that if $M, N, s_{1}, s_{2} \in \N$ 
such that $M \mid N$ and $0\leq s_{1},s_{2}\leq M-1$, then  
$$\#\{ 0 \le i \le N : i \equiv s_{1} \, (\text{mod } M) \} = \#\{ 1\le i \le N :  i\equiv s_{2} \, (\text{mod } M) \}  = 
\tfrac NM.$$ 
Since $k_\ell = \# S_\ell = \frac{q^m-1}{q^\ell+1} \in \N$ we have that $q^\ell+1 \mid q^{m}-1$. Thus, 
taking $M=q^\ell+1$ and $N=q^m-1$, from \eqref{T1} we get that 
$M_1 =M_2 = \frac{q^m-1}{q^\ell+1} = k_\ell$, as desired.

\medskip
\noindent (\textit{b}) 
Note that $k_\ell= \mu_\ell$ if and only if $\ell =\frac m2$ and $\varepsilon_\ell=-1$ (but this is automatic for $\ell= \frac m2$).
By \eqref{Autval}, it is clear that $k_{\frac m2} = \frac{q^m-1}{q^{\frac m2}+1} = q^{\frac m2}-1$ and $\upsilon_{\frac m2} = -1$. The multiplicity of $k_{\frac m2}$ is now $k_{\frac m2}+1=q^{\frac m2}$, showing ($b$).

\smallskip
It remains to show the integrality of the spectrum. 
For $\ell=\frac m2$ it is obvious. For $\ell\ne \frac m2$, we already know that $k_\ell \in \N$, hence 
it is enough to show that $\upsilon_\ell, \mu_\ell \in \Z$, i.e.\@ that 
$$q^\ell+1 \mid q^{\frac{m}{2}} - (-1)^{\frac{1}{2}m_{\ell}} \qquad \text{and} \qquad  
q^\ell+1 \mid q^{\frac{m}{2}+\ell} + (-1)^{\frac{1}{2}m_{\ell}}.$$ 
Note that since $m_{\ell}$ is even, then $\ell \mid \tfrac{m}{2}$.
Let $s,t$ be positive integers such that $s\mid t$. Since 
$q^s\equiv -1 \pmod{q^s+1}$ then $q^{t} = (q^s)^{\frac ts} \equiv (-1)^{\frac ts} \pmod{q^s+1}$.  
Thus, taking $t=\tfrac{m}{2}$ and $s=\ell$ we get that 
$q^{\frac{m}{2}}-(-1)^{\frac{1}{2}m_{\ell}} \equiv 0 \pmod{q^\ell+1}$, 
and multiplying this congruence by $q^\ell$, since $q^{\ell} \equiv -1 \pmod{q^\ell+1}$, we get that $\Gamma_{q,m}(\ell)$ has integral spectrum, and the result follows. 
\end{proof}

Notice that if we take $\ell=0$ in \eqref{Autval}, with the convention $\varepsilon_0=1$, 
we get the spectrum of the Paley graphs. 

\begin{note} \label{note1}
If $\tfrac 12 m_\ell$ is even we have 
$\lambda_2 = \upsilon_\ell$ and $\lambda_3 = \mu_\ell$ while if 
$\tfrac 12 m_\ell$ is odd 
$\lambda_2 = \mu_\ell$ and $\lambda_3 = \upsilon_\ell$. 
In both cases $\lambda_2 >0$ and $\lambda_3 <0$. 
\end{note}

We now give some relations on the eigenvalues and on the parameters of $\Gamma_{q,m}(\ell)$. 
Denote by $e_\ell$ the cardinality of the edge set $E_{q,m}(\ell)$ of $\Gamma_{q,m}(\ell)$.

\begin{coro} \label{prop eigens}
Let $k_\ell, \upsilon_\ell$ and $\mu_\ell$ be the eigenvalues of $\Gamma_{q,m}(\ell) \in \mathcal{G}_{q,m}$ as given in \eqref{Autval}. 

\begin{enumerate}[$(a)$]
\item We have $k_\ell = (\varepsilon_\ell \, q^{\frac m2} + 1) \upsilon_\ell$ and also $-q^\ell \upsilon_\ell = \mu_\ell +1$ 
and $(\upsilon_\ell, \mu_\ell)=1$.  \sk 

\item For fixed $q$ and $\ell$, we have $\upsilon_\ell(\Gamma_{2m+2\ell}) = \mu_\ell(\Gamma_{2m})$.  \sk 

\item If $\Gamma_{q,m}(d) \subset \Gamma_{q,m}(\ell)$ then  
$k_d \mid k_\ell$, $\upsilon_d \mid \upsilon_\ell$ and $e_d\mid e_\ell$.
\end{enumerate}
\end{coro}

\begin{proof}
To see (\textit{a}), since $(q^m-1) = (q^{\frac{m}{2}}-1)(q^{\frac{m}{2}}+1)$, we have 
$k_\ell = (\varepsilon_\ell \, q^{\frac m2}-1) \upsilon_\ell$, 
by \eqref{Autval}. 
The remaining assertions 
follows from the fact that $q^\ell \upsilon_\ell + \mu_\ell = -1$.
Item (\textit{b}) is clear from \eqref{Autval}.
For (\textit{c}), by Proposition \ref{teo subg} we have $\ell \mid d$ with $\frac{d}{\ell}$ odd and, by Lemma \ref{SLcSl}, 
$q^{\ell}+1 \mid q^{d}+1$. Hence 
$k_d \mid k_\ell$ and if $\varepsilon_\ell=\varepsilon_d$ also $\upsilon_d \mid \upsilon_\ell$. 
But, $\varepsilon_\ell=\varepsilon_d$ always holds since $\tfrac 12 m_\ell$ and $\tfrac 12 m_d$ have the same parity.
Finally, since $2e_d = k_d q^m$ and $k_\ell$ divides $k_d$ we get $e_d\mid e_\ell$. 
\end{proof}

The spectrum of the graphs considered is known and can also be obtained either by using Gauss sums in \eqref{chi gama 2}, by way of parameters of strongly regular graphs or by relating them with the weights of 2-weight cyclic codes (see for instance  \cite{BWX}, \cite{CK} or \cite{vLSch}). For completeness, we have provided an alternative proof using quadratic forms since we want to relate them to families of cyclic codes previously obtained by quadratic forms also (\cite{PV UMA}).		
	
\begin{rem}
In Section 4 of \cite{PV UMA} we have defined an irreducible cyclic code $\mathcal{C}_\ell$ determined by the quadratic forms $Q_{\gamma,\ell}(x)= \Tr_{q^m/q}(\gamma x^{q^\ell+1})$ where $\ell \mid m$ and $\gamma \in \ff_{q^m}^*$. Namely,
	$$\mathcal{C}_\ell = \{ c_\gamma = \big( \Tr_{q^m/q}(\gamma \alpha^{(q^\ell+1)i}) \big)_{i=0}^{n-1} : \gamma \in \ff_{q^m}\}
	\qquad \text{ with } \qquad n=\tfrac{q^m-1}{q^\ell+1}$$
where $\alpha$ is a primitive element in $\ff_{q^m}$. 
In Theorem 4.1 of \cite{PV UMA} we have found the parameters and weight distribution of $\mathcal{C}_\ell$. If $\ell \mid m$ then $\mathcal{C}_\ell$ is an $[n,m,d]_q$-code with 
$d=\tfrac{q^m-1}{q^\ell+1}(q-1)d'$, 
where $d'=q^{\frac m2 -1}$ if $\frac m2$ is even and $d'=q^{\frac m2 -1}-q^\ell$ if $\frac m2$ is odd. Also, $\mathcal{C}_\ell$ is $(p-1)$-divisible 2-weight code with weight distribution given by:
$$w_0=0, \qquad w_1= \tfrac{q-1}{q^\ell+1} (q^{m-1} - \varepsilon_\ell q^{\frac m2 -1}), \qquad 
	w_2= \tfrac{q-1}{q^\ell+1} (q^{m-1}+ \varepsilon_\ell q^{\frac m2 +\ell -1}), $$
with frequencies $1, n, nq^\ell$, respectively. It is easy to check that, if we enumerate the eigenvalues of $\Gamma_{q,m}(\ell)$ as
$\lambda_0 = k_\ell$, $\lambda_1=\nu_\ell$ and $\lambda_2=\mu_\ell$ (in the notation of \eqref{Autval}) then we have 
the simple linear relation 
$$ \lambda_i = n-\tfrac{q}{q-1} w_i, \qquad 1\le i \le 3,$$
where the multiplicities of $\lambda_i$ and the frequencies of $w_i$ coincide. 

This spectral relation holds in more generality (see \cite{PV0}) between the spectrum of GP-graphs 
\begin{equation} \label{Gkq}
		\Gamma(k,q)=X(\ff_q,\{x^k:x\in \ff_q^*\})
\end{equation}
and the weight distributions of the irreducible 
cyclic codes $\mathcal{C}(k,q) = \{ (\Tr_{q/p}(\gamma \omega^{ki}))_{i=0}^{n-1} : \gamma \in \ff_{q} \big \}$, where $\omega$ is a primitive element of $\ff_q$, $n= \frac{q-1}k$ and $q=p^m$. Note that $\Gamma_{q,m}(\ell)=\Gamma(q^{\ell+1},q^m)$.
\end{rem}

\section{Complements and structural properties} \label{sec:4}
Let $\bar \Gamma_{q,m}(\ell)$ be the complement of $\Gamma_{q,m}(\ell)$.
It is the Cayley graph $\bar \Gamma_{q,m}(\ell) = X(\ff_{q^m}, S_\ell^c\smallsetminus\{0\})$, where 
$S_\ell^c$ is the complement of $S_\ell$.
The complementary family of $\mathcal{G}_{q,m}$ is 
\begin{equation} \label{comp fam}
\bar{\mathcal{G}}_{q,m} = \{ \bar{\Gamma}_{q,m}(\ell) \,: \, \ell \mid m \text{ with } \ell \le [\tfrac m2] \text{ and } \tfrac{m}{\ell} \text{ even}\}.
\end{equation}

We begin by showing that for $\ell=\frac m2$ the graphs are disjoint union of copies of complete graphs and hence their complements
are complete multipartite graphs. We recall that the complete $k$-partite graph $K_{k\times m}$ is the graph whose vertex set is partitioned into $k$ independent sets of cardinal $m$, and there is an edge between every pair of vertices from different independent sets. 
\begin{lem} \label{completes}
We have $\Gamma_{q,m}(\frac m2) =  q^{\frac m2}K_{q^{\frac m2}}$ 
and $\bar\Gamma_{q,m}(\frac m2) = K_{q^{\frac m2} \times q^{\frac m2}}$. 
\end{lem}

\begin{proof}
Since $Spec(K_n) = \{ [n-1]^{1}, [-1]^{n-1} \}$, the graphs $\Gamma_{q,m}(\frac m2)$ and $q^{\frac m2}K_{q^{\frac m2}}$ 
have the same spectrum, by ($b$) of Theorem \ref{Spec Gamma}. 
But it is known that the graph $q^{\frac m2}K_{q^{\frac m2}}$ is uniquely determined by their spectrum (see for instance \cite{vDH}), implying that both graphs are the same. The remaining assertion is clear. 
\end{proof}

We now show that $\bar \Gamma_{q,m}(\ell)$ is the union of $q^\ell$ different Cayley graphs, all mutually isomorphic to 
$\Gamma_{q,m}(\ell)$. 
Let $\alpha$ be a primitive element of $\ff_{q^m}$. 
Then we have the cosets
$$\alpha^j S_\ell = \{\alpha^{i(q^\ell+1)+j}: 1\le i \le k = \tfrac{q^m-1}{q^\ell+1} \} \qquad \quad (0 \le j \le q^\ell).$$
Denote by $\Gamma_{q,m,\ell}^{(j)}$ the Cayley graph $X(\ff_{q^m}, \alpha^j S_\ell)$ and note that 
$\Gamma_{q,m,\ell}^{(0)} = \Gamma_{q,m}(\ell)$. Thus, $S_\ell$ is symmetric if and only if $\alpha^j S_\ell$ is symmetric, hence 
$\Gamma_{q,m,\ell}^{(j)}$ is undirected if and only if $\Gamma_{q,m}(\ell)$ is undirected.  

\begin{lem}
$\bar \Gamma_{q,m}(\ell) = \Gamma_{q,m,\ell}^{(1)} \cup \cdots \cup \Gamma_{q,m,\ell}^{(q^\ell)}$ and 
$\Gamma_{q,m,\ell}^{(j)} \simeq \Gamma_{q,m}(\ell)$ for every $j=1,\ldots,q^\ell$. 
\end{lem}

\begin{proof}
It is clear that $S_\ell^c \smallsetminus \{0\}$ equals the disjoint union 
$\alpha S_\ell \cup \alpha^2 S_\ell \cup \cdots \cup \alpha^{q^\ell} S_\ell$. Recall that $G =G_1\cup G_2$ is the graph with vertex set 
$V(G_1) \cup V(G_2)$ and edge set $E(G_1) \cup E(G_2)$. Hence,
$\bar \Gamma_{q,m}(\ell) = \Gamma_{q,m}^{(1)} \cup \cdots \cup \Gamma_{q,m}^{(q^\ell)}$ with 
$V(\Gamma_{q,m}^{(j)}) = V(\Gamma_{q,m}(\ell)) = \ff_q$ 
for every $j=1,\ldots, q^\ell$. To see that 
$\Gamma_{q,m}^{(j)} \simeq \Gamma_{q,m}^{(0)}$ note that the field automorphism $T_j(x) = \alpha^j x$ sends $S_\ell$ to $\alpha^j S_\ell$ and preserves edges, i.e.\@
if $x-y \in S_\ell$ then $T_j(x)-T_j(y) = \alpha^j(x-y) \in \alpha^jS_\ell$. 
\end{proof}

Note that for $\ell=0$ we have 
	$\bar \Gamma_{q,m}(0) = \Gamma_{q,m,0}^{(1)} = \Gamma_{q,m}(0)$, 
recovering the known fact that Paley graphs are self-complementary. 

We now compute the spectra of $\bar \Gamma_{q,m}(\ell)$ in terms of that of $\Gamma_{q,m}(\ell)$ and derive some structural consequences of the graphs from the spectra. 

\goodbreak 
\begin{prop} \label{Spec complement}
The spectrum of $\bar \Gamma_{q,m}(\ell)$ is integral and given as follows: 
\begin{enumerate}[$(a)$] 
\item If $\ell \ne \frac m2$, then the eigenvalues of $\bar \Gamma_{q,m}(\ell)$ are 
\begin{equation} \label{auts comp}
\bar{k}_\ell = \frac{q^\ell (q^m-1)}{q^\ell+1}, \qquad \quad
\bar{\upsilon}_{\ell} = \frac{-q^\ell ( \varepsilon_\ell q^{\frac m2-\ell} + 1)}{q^\ell+1}, \qquad \quad
\bar{\mu}_\ell = \frac{q^\ell ( \varepsilon_\ell q^{\frac m2 }-1)}{q^\ell+1},
\end{equation}
with corresponding multiplicities $1$, $q^\ell k_\ell$ and $k_\ell$, where 
$k_\ell$, $\upsilon_\ell$ are eigenvalues of $\Gamma_{q,m}(\ell)$ as given in Theorem \ref{Spec Gamma} and
$\varepsilon_\ell = (-1)^{\frac 12 m_\ell}$. \sk

\item If $\ell = \frac m2$, the eigenvalues of $\bar \Gamma_{q,m}(\frac m2)$ are 
$\bar k_{\frac m2} = q^{\frac m2}(q^{\frac m2}-1)$, $\bar{\upsilon}_{\frac m2} = 0$ and $\bar \mu_\ell = -q^{\frac m2}$, 
with multiplicities $1$, $q^{\frac m2}(q^{\frac m2}-1)$ and $q^{\frac m2}-1$, respectively.
\end{enumerate}
\end{prop}

\begin{proof}
(\textit{a}) If $A$ is the adjacency matrix of $\Gamma = \Gamma_{q,m}(\ell)$, then $J - I - A$ (where $J$ is the all 1's matrix) is the 
adjacency matrix of $\bar \Gamma$. Since $\Gamma$ is $k$-regular with $n$ vertices 
then $\bar \Gamma$ is $(n-k-1)$-regular. Thus $\bar k_\ell = n-k_\ell-1 = q^\ell (q^m-1) / (q^\ell+1)$. 
The remaining $n-1$ eigenvalues of $\bar \Gamma$ are of the form $-1-\lambda_i$ where $\lambda_i$ runs through the $n-1$ eigenvalues 
of $\Gamma$ belonging to an eigenvector orthogonal to $\mathbf{1}$. That is, $\bar \upsilon_\ell = -1-\upsilon_\ell$ and 
$\bar\mu_\ell = -1-\mu_\ell$. The result follows by \eqref{Autval} in Theorem~\ref{Spec Gamma}. 
Moreover, since $\Gamma_{q,m}(\ell)$ is integral by Theorem \ref{Spec Gamma}, $\bar \Gamma_{q,m}(\ell)$ is also integral.

\noindent 
(\textit{b}) It is known that $Spec(K_{a \times a}) = \{[a(a-1)]^{1}, [0]^{a(a-1)}, [-a]^{a-1} \}$, hence it is integral. 
By Lemma \ref{completes}, taking $a=q^{\frac m2}$ we get the result.
\end{proof}

\begin{note} \label{note2}
If $\tfrac 12 m_\ell$ is even we have 
$\bar \lambda_2 = \bar \mu_\ell$ and $\bar \lambda_3 = \bar \upsilon_\ell$ while if 
$\tfrac 12 m_\ell$ is odd 
$\bar \lambda_2 = \bar \upsilon_\ell$ and $\bar \lambda_3 = \bar \mu_\ell$. 
In both cases $\lambda_2 >0$ and $\lambda_3 <0$. 
\end{note}

\begin{coro} \label{prop eigen comps}
The eigenvalues of $\bar \Gamma_{q,m}(\ell)$ are multiples of $q^\ell$ and we have 
$\bar{k}_\ell = q^\ell k_\ell$, $\bar{\mu}_\ell = q^\ell \upsilon_\ell$ and $\bar{\upsilon}_\ell +1 = \frac{\mu_\ell+1}{q^\ell}$. Moreover, 
$(\frac{\bar \upsilon_\ell}{q^\ell}, \frac{\bar \mu_\ell}{q^\ell})=1$ and $\upsilon_\ell + \bar \mu_\ell = 
\bar \upsilon_\ell + \mu_\ell=-1$. 
\end{coro}
\begin{proof}
By \eqref{auts comp} implies that $q^\ell$ divides $\bar k_\ell, \bar \upsilon_\ell$ and $\bar \mu_\ell$. Also, by \eqref{Autval} we have 
$\bar{k}_\ell = q^\ell k_\ell$ and $\bar{\mu}_\ell = q^\ell \upsilon_\ell$. 
Finally, note that 
$q^\ell \big( \frac{\bar \upsilon_\ell}{q^\ell} \big) + \big( \frac{\bar \mu_\ell}{q^\ell} \big) = -1$,
from which the result follows.
\end{proof}

As a direct consequence of Theorem \ref{Spec Gamma}, we will show that all the graphs considered 
are non-bipartite and primitive (connected with connected complement), with trivial exceptions. 
For $\ell \ne \frac m2$ the graphs are primitive. 

\begin{prop} \label{properties}
For fixed $q$ and $m$, both $\mathcal{G}_{q,m}$ and $\bar{\mathcal{G}}_{q,m}$ are finite families of 
mutually 
non-isospectral graphs. Moreover: 
\begin{enumerate}[$(a)$]
\item $\Gamma_{q,m}(\ell)$ and $\bar\Gamma_{q,m}(\ell)$ are connected for every $\ell \ne \frac m2$ (primitiveness). \sk 

\item $\Gamma_{q,m}(\frac m2)$ is disconnected while $\bar\Gamma_{q,m}(\frac m2)$ is connected. \sk 

\item $\Gamma_{q,m}(\ell)$ and $\bar\Gamma_{q,m}(\ell)$ are non-bipartite
$($except for $\Gamma_{2,2}(1) = 2K_2$, $\bar \Gamma_{2,2}(1) = C_4 = K_{2,2})$.
\end{enumerate}
\end{prop}

\begin{proof}
By Theorem \ref{Spec Gamma} and Proposition \ref{Spec complement}, if $\ell,\ell'$ are different divisors of $m$ the graphs 
$\Gamma_{q,m}(\ell)$ and $\Gamma_{q,m}(\ell')$ are non-isospectral. The same holds for the pairs
$\bar \Gamma_{q,m}(\ell)$, $\bar \Gamma_{q,m}(\ell')$ and $\Gamma_{q,m}(\ell)$, $\bar \Gamma_{q,m}(\ell')$.

\noindent (\textit{a}) $\Gamma_{q,m}(\ell)$ (resp.\@ $\bar \Gamma_{q,m}(\ell)$) is connected if and only if $k_\ell$ 
(resp.\@ $\bar k_\ell$) has multiplicity 1; and this happens if and only if 
$\ell \ne \frac m2$, by Theorem \ref{Spec Gamma} (resp.\@ Proposition \ref{Spec complement}). Thus, $\Gamma_{q,m}(\ell)$ and 
$\bar \Gamma_{q,m}(\ell)$ are connected for $\ell \ne \frac m2$. 

\noindent (\textit{b}) 
This follows directly by Lemma \ref{completes}. 

\noindent (\textit{c}) The graphs $\Gamma_{q,m}(\ell)$ and $\bar \Gamma_{q,m}(\ell)$ are non-bipartite since $-k$ is not an eigenvalue, unless $q=m=2$ and $\ell=1$. In these cases we have 
$Spec(\Gamma_{2,2}(1)) = Spec(2K_2) = \{ [1]^{2}, [-1]^{2} \}$
(notice that $k=1$ and $\mu=1$ correspond to different eigenvectors) and 
$Spec(\bar\Gamma_{2,2}(1)) = Spec(2K_2) = \{ [2]^{1}, [0]^{2}, [-2]^{1}\}$.
\end{proof}

For a graph $\Gamma$, the number $w_r(\Gamma)$ of closed walks of length $r$ in $\Gamma$ 
can be expressed in terms of the spectrum. 
We now compute these numbers for the graphs $\Gamma_{q,m}(\ell)$ and their complements.
This can be used to give the number $c_n(\Gamma)$ of $n$-cycles of $\Gamma$, since we clearly have
$c_r(\Gamma)=\tfrac{1}{2r}w_r(\Gamma)$.

\begin{coro} \label{coro walks}
We have that 
\begin{align} \label{walks}
\begin{split}
w_r(\Gamma_{q,m}(\ell)) & = \begin{cases} 
k \, \{ k^{r-1} + q^\ell \upsilon^r + \mu^r \}, & \qquad \qquad \qquad \qquad \quad  \text{ if } \ell \ne \frac m2, \\[1.25mm]
q^m k \, \sum\limits_{j=1}^{r-1} \tbinom{r-1}{j} q^{\frac m2 (j-1)},  & \qquad \qquad \qquad \qquad \quad \text{ if } \ell = \frac m2,
\end{cases}
\\[1.35mm]
w_r(\bar \Gamma_{q,m}(\ell)) & = \begin{cases} 
q^\ell k \, \{ (q^\ell k)^{r-1} + (-1-\upsilon)^r + q^{\ell (r-1)} \upsilon^r \}, & \quad \:\:\: \text{ if } \ell \ne \frac m2,
 \\[1.25mm]
(q^{\frac m2})^{r} k \, \sum\limits_{j=1}^{r-1} \tbinom{r-1}{j} q^{\frac m2 (j-1)},  & \quad \:\:\: \text{ if } \ell = \frac m2,
\end{cases}
\end{split}
\end{align}
where $k=k_\ell$, $\upsilon=\upsilon_\ell$, $\mu=\mu_\ell$ and $\bar k = \bar k_\ell$, $\bar \upsilon = \bar \upsilon_\ell$, 
$\bar \mu = \bar \mu_\ell$ are the eigenvalues of $\Gamma_{q,m}(\ell)$ and $\bar \Gamma_{q,m}(\ell)$, respectively.
In particular, $k \mid w_r(\Gamma_{q,m}(\ell))$ and $\bar k \mid w_r(\bar \Gamma_{q,m}(\ell))$.
\end{coro}

\begin{proof}
It is well known that the number of closed walks of length $r$ is given by the $r$-th moment sums 
	$$w_r(\Gamma) = \lambda_1^r + \cdots + \lambda_n^r.$$ 
Thus, by Theorem \ref{Spec Gamma}, we get
$w_r(\Gamma_{q,m}(\ell)) = k_\ell^r+q^\ell k_\ell \upsilon_\ell^r + k_\ell \mu_\ell^r $
for $\ell \ne \tfrac m2$ and also
$$w_r(\Gamma_{q,m}(\tfrac m2)) = q^{\frac m2} (q^{\frac m2}-1)^r + q^{\frac m2} (q^{\frac m2}-1) (-1)^r = 
k_{\frac m2} q^{\frac m2} \{ (q^{\frac m2}-1)^{r-1} +(-1)^r \}$$
from which the first equality in \eqref{walks} follows. For the complementary graphs we proceed similarly using 
Proposition~\ref{Spec complement} and we obtain the second identity in \eqref{walks}.
The remaining assertions are clear.
\end{proof}

\begin{exam} \label{triangles}
Let us compute the girths of the graph $\Gamma=\Gamma_{2,4}(1)$ and of its complement $\bar \Gamma$. 
These graphs have eigenvalues $\{5, 1, -3 \}$ and $\{10, 2, -2 \}$, respectively. By the previous 
corollary the number of $r$-cycles in $\Gamma$ is given by
$$c_r (\Gamma) = \tfrac{1}{2r} w_r(\Gamma) = \tfrac{1}{2r} k(k^{r-1}+2 \upsilon^r+\mu^r) = \tfrac{5}{2r} (5^{r-1} +2+(-3)^r).$$
Thus $c_3 (\Gamma) = \tfrac 56 (5^2 +2+(-3)^3)=0$ and $c_4 (\Gamma) = \tfrac 58 (5^3 +2+(-3)^4)=30$ and hence 
the girth of $\Gamma$ is $g(\Gamma) = 4$. Also, 
$c_3 (\bar \Gamma) = \tfrac 16 w_3(\bar \Gamma) = 
\tfrac 53 (10^2 +2^3 +4(-2)^3)=100$, so $g(\bar \Gamma)=3$. \hfill $\lozenge$
\end{exam}

\section{Strongly regular graphs, Latin squares, and invariants} \label{sec:5}
A \textit{strongly regular graph} (SRG) with parameters $v,k,e,d$ (often called $n,k,\lambda, \mu$), denoted $srg(v,k,e,d)$, is a $k$-regular graph with $v$ vertices such that for any pair of vertices $x,y$ the number of vertices adjacent (resp.\@ non-adjacent) to both $x$ and $y$ is $e\ge 0$ 
(resp.\@ $d\ge 0$). These parameters are tied by the relation 
\begin{equation} \label{main relation} 
(v-k-1)d = k(k-e-1). 
\end{equation} 
Strongly regular graphs are, in particular, distance regular graphs and have diameter $\delta=2$ if $d\ne 0$.
For instance, the complete graphs $K_n$ are $srg(n,n-1,n-2,0)$ for $n\ge 3$. Also, Paley graphs 
$P(q^m)$ with $q^m\equiv 1 \pmod 4$ are strongly regular with parameters 
\begin{equation} \label{srg Pq} 
\big(q^m, \, \tfrac 12(q^m-1), \, \tfrac 14 (q^m-5), \, \tfrac 14 (q^m-1) \big),
\end{equation} 
so they belong to the half case, i.e.\@ those with parameters 
$(4t+1,2t,t-1,t)$ for some $t$.

\subsubsection*{Spectrum}
The spectrum of strongly regular graphs is well understood. 
Every strongly regular graph $\Gamma=srg(v,k,e,d)$ has 3 eigenvalues given by 
\begin{equation} \label{srg auts}
\lambda_1=k \qquad \quad \text{and} \quad \qquad \lambda_2^{\pm} = \tfrac 12 \{(e-d) \pm \Delta \}
\end{equation}
where $\Delta =  \sqrt{(e-d)^2+4(k-d)}$,
with corresponding multiplicities 
$$m_1=1 \qquad \text{and} \qquad m_2^{\pm} = \tfrac 12 \{(v-1) \mp \tfrac{2k+(v-1)(e-d)}{\Delta}\}.$$
If $\Gamma$ is connected the converse also holds (see for instance \cite{BH}); that is, if $\Gamma$ is connected and has 
3 eigenvalues, then it is a strongly regular graph.

\smallskip
$\Gamma$ is called a \textit{conference} graph if
\begin{equation} \label{conf} 
2k+(v-1)(e-d)=0
\end{equation} 
(for instance Paley graphs are conference graphs). On the other hand, if $2k+(v-1)(e-d) \ne 0$ then  
$\Gamma=srg(v,k,e,d)$ has integral spectrum with unequal multiplicities. 

The condition for 
$\Gamma=srg(v,k,e,d)$ to be Ramanujan (see \eqref{rama}) is 
\begin{equation} \label{srg ram}
(e-d) + \sqrt{(e-d)^2+4(k-d)} \le 4\sqrt{k-1}
\end{equation}
Note that if $e=d$ then $\lambda_2^- = -\lambda_2^+$ (with multiplicity $v-1$) and \eqref{srg ram} is then equivalent to $3k+d\ge 4$ which always holds, except for the trivial case $k=1$ and $d=0$.  

Therefore, SRGs 
with $e=d$ are trivially integral, non-bipartite and Ramanujan graphs. 
The smallest examples of such graphs are: the triangular graph $T(6)$, the Shrikhande graph and the complement of the Clebsch graph with parameters respectively given by 
$$(15,8,4,4), \qquad (16,6,2,2), \qquad \text{and} \qquad (16,10,6,6).$$

\begin{rem}
If $f:\ff_2^n \rightarrow \ff_2$ is a Boolean function, the Cayley graph $\Gamma_f = X(\ff_2^n, \Omega_f)$ with $\Omega_f = \{x\in \ff_2^n : f(x)=1 \}$ is considered in \cite{St}. 
It is proved that $\Gamma_f$ is an SRG 
with the property $e=d$ if and only if $f$ is bent.
Thus, 
	$\{\Gamma_f \,|\, f:\ff_2^n \rightarrow \ff_2 \text{ bent}\}_{n\in \N}$ 
is an infinite family of integral non-bipartite strongly regular Ramanujan graphs over $\ff_2$ with $e=d$. In Section 8 we will construct infinite families of integral non-bipartite strongly regular Ramanujan graphs having $e \ne d$ over any field.
\end{rem}

\subsection{Parameters as strongly regular}
We now show that
$\Gamma_{q,m}(\ell)$ and $\bar \Gamma_{q,m}(\ell)$ are SRGs with $e\ne d$. The complement of 
$\Gamma = srg(v, k, e, d)$ is also an SRG,
$\bar \Gamma = srg(v, \bar k, \bar  e, \bar d)$, with parameters 
\begin{equation} \label{srg comp} 
 \bar k = v-k-1, \quad \qquad \bar  e = v-2-2k+d, \quad \qquad \bar d = v-2k+e.
\end{equation} 
Notice that any Paley graph is self-complementary. 
We now give the parameters as SRGs for the graphs in $\mathcal{G}_{q,m}$ and $\bar{\mathcal{G}}_{q,m}$.
In particular, no graph $\mathcal{G}_{q,m}$ in $\mathcal{G}_{q,m}$ is self-complementary.

\goodbreak 

\begin{thm} \label{prop srg}
All the graphs $\Gamma_{q,m}(\ell) \in \mathcal{G}_{q,m}$ and $\bar \Gamma_{q,m}(\ell) \in \bar{\mathcal{G}}_{q,m}$ 
are strongly regular, not of conference type. 

\begin{enumerate}[$(a)$] 
\item If $\ell \ne \frac m2$, then $\Gamma_{q,m}(\ell) = srg(q^m, k, e, d)$ and 
$\bar \Gamma_{q,m}(\ell) = srg(q^m, q^\ell k, \bar e, \bar d)$
with $k=\frac{q^m-1}{q^{\ell}+1}$, 
\begin{align} \label{eded}
\begin{split}
e &= \tfrac{q^m - \varepsilon_\ell \, q^{\frac{m}{2} + \ell}(q^{\ell}-1) - 3q^{\ell} - 2}{(q^{\ell}+1)^{2}}, 
\qquad \:\:\:\, d = \tfrac{q^m + \varepsilon_\ell \, q^{\frac{m}{2}}(q^{\ell}-1)-q^{\ell}}{(q^{\ell}+1)^{2}}, \\[1.25mm]
\bar e &= \tfrac{q^{2\ell}(q^m -2) + \varepsilon_\ell \, q^{\frac m2} (q^{\ell}-1) - 3q^{\ell}}{(q^{\ell}+1)^{2}}, 
\qquad \bar d = \tfrac{q^\ell (q^{m+\ell}-1) - \varepsilon_\ell \, q^{\frac m2 + \ell} (q^{\ell}-1)}{(q^{\ell}+1)^{2}},
\end{split}
\end{align}
where $\varepsilon_\ell = (-1)^{\frac 12 m_\ell}$. 

\item If $\ell = \frac m2$, then $\Gamma_{q,m}(\frac m2) = srg(q^m, k, k-1, 0)$ and 
$\bar \Gamma_{q,m}(\frac m2) = srg(q^m, q^{\frac m2} k, q^{\frac m2} (k-1), q^{\frac m2} k)$ with 
$k= q^{\frac m2}-1$ $($except for $\Gamma_{2,2}(1)=2K_2)$.
\end{enumerate}
In particular, $\mathcal{G}_{q,m} \cap \bar{\mathcal{G}}_{q,m} = \varnothing$. 
\end{thm}

\begin{proof}
If $\ell \ne \tfrac m2$, then $\Gamma_{q,m}(\ell)$ is connected by Proposition \ref{properties} and hence, since it has 3 eigenvalues, it is a strongly regular graph. If $\ell = \tfrac m2$, by Proposition \ref{properties} the graph $\Gamma_{q,m}(\ell)$ is a disjoint union of complete graphs, and hence strongly regular also. The complement of an SRG is an SRG.
It is straightforward to check that \eqref{conf} holds if and only if $\ell=0$, but $\ell=0$ is not allowed for the GP-graphs in the families $\mathcal{G}_{q,m}$ and $\bar{\mathcal{G}}_{q,m}$ (see \eqref{family} and \eqref{comp fam}).

\noindent (\textit{a}) 
We already know that $\Gamma_{q,m}(\ell)$ is a $srg(v,k,e,d)$ with $v = q^m$ and $k = \frac{q^m-1}{q^{\ell}+1}$. 
To compute $e$ and $d$ note that by \eqref{srg auts} we have 
$$e= k + \lambda_2^+ \lambda_2^- + \lambda_2^+ + \lambda_2^- \qquad \text{ and } \qquad d=k+\lambda_2^+\lambda_2^-,$$ 
where $\lambda_2^{\pm}$ are the non-trivial eigenvalues $\upsilon_\ell$ and $\mu_\ell$ in Theorem \ref{Spec Gamma}.
Solving for $e$ and $d$ in the above expressions we get the parameters for $\Gamma_{q,m}(\ell)$. The parameters of the complementary graph follow from \eqref{srg comp}. 

\noindent (\textit{b}) 
It is known that $aK_n = srg(an, n-1,n-2,0)$ and $K_{a \times n} = srg(an, (a-1)n, (a-2)n, (a-1)n)$.
The result follows by taking $a=n=q^{\frac m2}$, except for $\Gamma_{2,2}(1)$ since the parameters $(4,1,0,0)$ has no sense.

The remaining assertion is clear after some straightforward computations.
\end{proof}

\begin{rem} \label{inf fam}
(\textit{i}) The above  proposition gives two infinite families of strongly regular graphs with explicit parameters. 
In particular, $e, d >0$, $\bar e, \bar d>0$ and $e \ne d$, $\bar e \ne \bar d$ in all cases with the only exceptions of
the Clebsch graph $\Gamma_{2,4}(1)=srg(16,5,0,2)$ and its complement $\bar \Gamma_{2,4}(1) = srg(16,10,6,6)$, respectively. 
In particular, this proves that the Clebsch graph is the \textsl{only} graph in our families without triangles. 

\noindent 
(\textit{ii}) By taking $\ell=0$ in \eqref{eded} we recover the parameters in \eqref{srg Pq} for Paley graphs.

\noindent
(\textit{iii}) In \cite{MTR}, primitive strongly regular graphs are classified into 3 classes.
If $\Gamma = srg(n,k,e,d)$ and $\Gamma$ is not the Clebsch graph or its complement 
(i.e.\@ $\Gamma \ne \Gamma_{2,2}(1), \bar \Gamma_{2,2}(1)$) then either
$$ (\text{\textit{a}}) \quad n\le \min \big\{ \tfrac{f(f+1)}{2}, \tfrac{g(g+1)}{2} \big\} \qquad \text{ or } \qquad 
(\text{\textit{b}}) \quad n = \min \big \{ \tfrac{f(f+3)}2, \tfrac{g(g+3)}2 \big \},$$ 
where $f,g$ are the multiplicities of the non-trivial eigenvalues. 
It is easy to check that all the graphs $\Gamma_{q,m}(\ell)$ and $\bar \Gamma_{q,m}(\ell)$ --with $(q,m,\ell) \ne (2,2,1)$-- 
satisfy condition ($a$) strictly.
\end{rem}

A connected strongly regular graph $\Gamma$, being a distance regular graph of diameter $\delta =2$, has an intersection array of the form $\mathcal{A}(\Gamma) = \{b_0,b_1, b_2; c_0, c_1, c_2\}$. For every $i=0,1,2$ and every pair of vertices $x,y$ at distance $i$, the intersection numbers are defined by 
$$b_i = \#\{z\in N(y): d(x,z)=i+1 \} \qquad \text{and} \qquad c_i = \#\{ z\in N(y) : d(x,z)=i-1 \},$$ 
where $N(y)$ denotes the set of neighbors of $y$. Since we trivially have $b_2=c_0=0$, we will simply write $\mathcal{A}(\Gamma) = \{b_0,b_1;  c_1, c_2\}$, as it is usual.
\black 

\begin{coro} \label{coro array}
If $\ell \ne \frac m2$, the intersection arrays of $\Gamma=\Gamma_{q,m}(\ell)$ and $\bar \Gamma=\bar \Gamma_{q,m}(\ell)$ are given by
\begin{equation} \label{arrays} 
\begin{split}
& \mathcal{A}(\Gamma) = \{ k, k-e-1; 1, d \} = \{k, q^\ell d; 1, d\}, \\[1mm]
& \mathcal{A}(\bar \Gamma) = \{ v-k-1, k-d; 1, v-2k+e \} = \{ q^\ell k, k-d; 1, q^\ell(e+d)\},
\end{split}
\end{equation}
where $k$, $d$ and $e$ are given in Theorem \ref{prop srg}.
\end{coro}
\begin{proof}
We know that $\Gamma = srg(v,k,e,d)$ is primitive for $\ell \ne \frac m2$. 
Since $\Gamma = srg(v,k,e,d)$ is connected with $\delta=2$, its intersection array is $\{ k, k-e-1; 1, d \}$. 
In fact, it is clear that $b_0=k$ and $c_1=1$. Let $x,y$ be vertices of $\Gamma$. 
Thus, if $d(x,y)=1$, then 
$$b_1 = \#(N(y) \smallsetminus\{x\}) - \#N(x) =k-1-e.$$ 
If $d(x,y)=2$ then 
$c_2 = \#(N(x) \cap N(y))$. 
Since $\bar \Gamma$ is also connected with diameter 2, its intersection array is $\{ \bar k, \bar k - \bar e-1; 1, \bar d \}$.
Now, since $v-2k+e = (v-k+1)-(k-e+1)$, by using \eqref{main relation} and \eqref{srg comp} we get the desired result.
\end{proof}

\subsubsection*{Latin square type graphs}
For the definitions and results in this subsection we refer to \cite{CvL}.
A strongly regular graph $\Gamma$ is \textit{geometric} if it is the point graph of a partial geometry $pg(\sigma,t,\alpha)$. 
In this case, $\Gamma$ has parameters (see Proposition 7.3 in \cite{CvL})
\begin{equation} \label{pg}
srg( \tfrac{1}{\alpha} (\sigma+1)(\sigma t+\alpha), (t+1)\sigma, \sigma-1+t(\alpha-1), (t+1)\alpha).
\end{equation}

A graph is \textit{pseudo-geometric} if it has the parameters in \eqref{pg}. 
Strongly regular graphs with parameters 
\begin{equation} \label{PLS}
\begin{aligned}
PL_{t}(u) = srg(u^2, t(u-1), t^2-3t+u, t(t-1)), \\[1.25mm]
NL_{t}(u) = srg\big( u^2, t(u+1), t^2+3t-u, t(t+1) \big),
\end{aligned}
\end{equation}
for some $t,u>0$, are respectively called \textit{pseudo Latin square graph} and \textit{negative Latin square graph} (notice that $NL_{t}(u)=PL_{-t}(-u)$). If a pseudo Latin square graph $PL_t(u)$ is geometric then it is called \textit{Latin square graph}
and it is denoted $L_t(u)$. In particular, if $t=2$, $L_2(u)$ is called a \textit{lattice} graph.

We now show that the connected graphs $\Gamma_{q,m}(\ell)$ in $\mathcal{G}_{q,m}$ are pseudo-geometric. 
\begin{prop} \label{teo latin}
Any $\Gamma_{q,m}(\ell)$ in $\mathcal{G}_{q,m}$ with $\ell \ne \frac m2$ is pseudo-geometric. More precisely: 
\begin{enumerate}[$(a)$]
\item If $\tfrac 12 m_\ell$ is odd then $\Gamma = PL_{t}(q^{\frac m2})$. \msk  
	
\item If $\tfrac 12 m_\ell$ is even then $\Gamma = NL_{t}(q^{\frac m2})$. 
\end{enumerate}	
Here $t=\tfrac{q^{\frac m2}-\varepsilon_{\ell}}{q^\ell+1}=|\upsilon_\ell|$ and $\varepsilon_{\ell}=(-1)^{\frac 12 m_{\ell}}$, where $\upsilon_\ell$ is one of the eigenvalues of $\Gamma$ in \eqref{Autval}. 
\end{prop}

\begin{proof}
Let $\Gamma = \Gamma_{q,m}(\ell)$ with $0 < \ell < \frac m2$ and $\ell \mid m$. Note that the regularity degree of $\Gamma$ equals the multiplicity of a non-trivial eigenvalue, namely $k=m(\mu_\ell)$, by ($a$) of Theorem \ref{Spec Gamma}. 
Thus, by Proposition 8.14 in \cite{CvL}, $\Gamma$ is either a pseudo Latin square graph, a negative Latin square graph or a conference graph. 
Since $\Gamma$ is not a conference graph, by ($a$) in Theorem~\ref{prop srg}, 
then $\Gamma$ is a pseudo Latin square graph or a negative Latin square graph, hence with parameters as in \eqref{PLS}.

Recall that $\Gamma=srg(n,k,e,d)$ with $n=q^m$, $k=\tfrac{q^m-1}{q^\ell+1}$ and $e$, $d$ as in \eqref{eded}. 
Take $u=|\mu_\ell - \upsilon_\ell|$ and $t=|\upsilon_\ell|$. 
One can check that $u=q^{\frac m2}$ for $\tfrac 12 m_\ell$ even or odd, and hence $n=u^2$.
It is clear that $k=t(u-1)$ holds (resp.\@ $k=t(u+1)$) if and only if $\varepsilon_\ell=-1$ (resp.\@ $\varepsilon_\ell=1$), 
that is if and only if $\tfrac 12 m_\ell$ is odd (resp.\@ even). It is straightforward to check that $t^2-3t+u=e$ and $t(t-1)=d$
(resp.\@ $t^2+3t-u=e$ and $t(t+1)=d$) for $\tfrac 12 m_\ell$ is odd (resp.\@ even).
Hence, $\Gamma=PL_t(u)$ if $\tfrac 12 m_\ell$ is odd and $\Gamma=NL_t(u)$ if $\tfrac 12 m_\ell$ is even, proving 
($a$) and ($b$).
\end{proof}

\subsection{Invariants} \label{sec:6a}
For completeness, given a graph $\Gamma$ in our families, we now summarize the exact values, or bounds, for some classical invariants (such as the diameter $\delta(\Gamma)$ and the girth $\gamma(\Gamma)$, the independence number $\alpha(\Gamma)$, the clique number $\omega(\Gamma)$, the chromatic number $\chi(\Gamma)$, the vertex, edge and algebraic connectivities $\kappa(\Gamma)$, $\varepsilon(\Gamma)$, and $\theta_2(\Gamma)$ respectively, and the isoperimetric constant $h(\Gamma)$), for $\Gamma$ and $\bar \Gamma$ in terms of the spectrum of $\Gamma$.

\begin{prop} \label{prop invariants}
Let $\Gamma = \Gamma_{q,m}(\ell) \in \mathcal{G}_{q,m}$ and $\bar \Gamma = \bar \Gamma_{q,m}(\ell) \in \bar{\mathcal{G}}_{q,m}$. 
Then, we have:
\begin{enumerate}[$(a)$] \setlength\itemsep{1.2mm}
\item The diameters of $\Gamma$ and $\bar \Gamma$ are $\delta(\Gamma) = \delta(\bar \Gamma) = 2$. \sk 

\item The girths of $\Gamma$ and $\bar \Gamma$ are $\gamma(\Gamma) = \gamma(\bar \Gamma) = 3$, except for $\gamma(\Gamma_{2,4}(1)) = 4$.
\sk 

\item The independence and clique numbers of $\Gamma$ and $\bar \Gamma$ satisfy: \sk

\begin{enumerate}[$(i)$]
	\item If  \small{$\frac 12$}$m_{\ell}$ is even then	
		$$ \omega( \Gamma) = \alpha(\bar\Gamma) \le \frac{q^m+q^{\frac{m}{2}+\ell}}{q^{\frac{m}{2}+\ell}+1} \qquad \text{and} \qquad 
		\omega(\bar \Gamma) = \alpha(\Gamma) \le \frac{q^m+q^{\frac{m}{2}-\ell}}{q^{\frac{m}{2}-\ell}+1}.$$ 
		
\item If  \small{$\frac 12$}$m_{\ell}$ is odd then 
$\omega(\Gamma)=\alpha(\bar\Gamma)=\omega(\bar \Gamma) =\alpha(\Gamma)=q^{\frac m2}$. 
	\end{enumerate}
\sk

\item The chromatic numbers of $\Gamma$ and $\bar \Gamma$ satisfy: \sk

\begin{enumerate}[$(i)$]
\item If  \small{$\frac 12$}$m_{\ell}$ is even then
$$\chi(\Gamma)\ge \frac{(q^m-1)(q^{\frac{m}{2}-\ell}+1)}{q^m+q^{\frac{m}{2}-\ell}} \qquad \text{and}\qquad 
\chi(\bar \Gamma)\ge\frac{(q^m-1)(q^{\frac{m}{2}+\ell}+1)}{q^m+q^{\frac{m}{2}+\ell}}.$$
		
\item	If  \small{$\frac 12$}$m_{\ell}$ is odd then 
$\chi(\Gamma)=\chi(\bar \Gamma)= q^{\frac{m}{2}}$.
\end{enumerate}
\sk

\item The vertex connectivity and edge connectivity of $\Gamma$ and $\bar \Gamma$ are given by
$\kappa(\Gamma) = \varepsilon(\Gamma) = \tfrac{q^m-1}{q^\ell+1}$ and 
$\kappa(\bar \Gamma) = \varepsilon(\bar \Gamma) = \tfrac{q^\ell(q^m-1)}{q^\ell+1}$.
In particular, $\kappa(\Gamma) =q^\ell \kappa(\bar \Gamma)$ and $\varepsilon(\Gamma) =q^\ell \varepsilon(\bar \Gamma)$.  

\sk

\item The algebraic connectivity of $\Gamma$ and $\bar \Gamma$ are given by
$$ \theta_2(\Gamma) = \begin{cases} 
\upsilon_\ell & \:\: \text{if \small{$\frac 12$}$m_\ell$ is even}, \\
\mu_\ell      & \:\: \text{if \small{$\frac 12$}$m_\ell$ is odd},
\end{cases}
\qquad \text{and} \qquad  
\theta_2(\bar \Gamma) = \begin{cases} 
\bar \mu_\ell        & \:\: \text{if \small{$\frac 12$}$m_\ell$ is even}, \\
\bar \upsilon_\ell   & \:\: \text{if \small{$\frac 12$}$m_\ell$ is odd}.
\end{cases}$$

\sk 

\item The isoperimetric constants of $\Gamma$ and $\bar \Gamma$ satisfy: \sk 

\begin{enumerate}[$(i)$]
	\item If  \small{$\frac 12$}$m_{\ell}$ is odd, then 
	$$ \tfrac 12 (k_{\ell} - \upsilon_{\ell}) \le h(\Gamma) \le \sqrt{k_{\ell}^2-\upsilon_{\ell}^2} 
	\qquad \text{and} \qquad  \tfrac 12 (q^{\ell} k_{\ell} - \bar\mu_{\ell}) \le h(\bar \Gamma) \le
	\sqrt{q^{2\ell} k_{\ell}^2 - \bar{\mu}_{\ell}^2} .$$					
	
	\item If  \small{$\frac 12$}$m_{\ell}$ is even, then 
	$$\tfrac 12 (k_\ell - \mu_\ell) \le h(\Gamma) \le \sqrt{k_\ell^2 - \mu_\ell^2} 
	\qquad \text{and} \qquad \tfrac 12 (q^{\ell}k_{\ell} - \bar\upsilon_{\ell}) \le h(\bar \Gamma) \le 
	\sqrt{q^{2\ell} k_{\ell}^2 - \bar\upsilon_{\ell}^2} .$$
\end{enumerate}
\end{enumerate}
In items $(a)$, $(e)$ and $(f)$ one has to take $\ell\neq\tfrac m2$.
\end{prop}

\noindent \textit{Note.} In the proposition, the case $\ell=\frac m2$ is contained in the case $\frac 12 m_\ell$ odd.

\begin{proof}
Since $\Gamma_{q,m}(\tfrac m2)$ is disconnected, we exclude the cases $\ell=\tfrac m2$ in items ($a$), $(e)$ and $(f)$.
  	
To prove (\textit{a}) and ($b$) note that, by Theorem \ref{prop srg}, every $\Gamma$ considered is an $srg(v,k,e,d)$ with $e>0$ (except for $\Gamma=\Gamma_{2,4}(1)$) and $d>0$ (except when $\ell =\tfrac m2$), hence $\delta(\Gamma)=2$ and $\gamma(\Gamma)=3$. 
We already saw in Example \ref{triangles} that $\gamma(\Gamma_{2,4}(1)) = 4$  
(alternatively, $\Gamma_{2,4}(1)=srg(16,5,0,2)$, and the values $e=0$ and $d=2$ clearly imply that $\gamma(\Gamma)=4$).

To prove (\textit{c}), first recall that, by definitions, the clique number of a graph $X$ (resp.\@ $\bar X$) is the independence number of its complement $\bar X$ (resp.\@ $X$), so $\omega(\Gamma) = \alpha(\bar\Gamma)$ and $\omega(\bar \Gamma) = \alpha(\Gamma)$. Now, we use the known bound 
$\omega(X) \le 1-\tfrac k\mu$, where $X$ is a SRG with degree of regularity $k$ and $\mu < 0$ is the smallest eigenvalue of $X$. 
By Notes \ref{note1} and \ref{note2}, if $\frac{1}{2}m_{\ell}$ is even, then  $\mu=\mu_{\ell}=-\frac{q^{\frac{m}{2}+\ell}+1}{q^{\ell}+1}$ and 
$\mu = \bar{\upsilon}_{\ell} = \tfrac{-q^\ell ( q^{\frac m2-\ell} + 1)}{q^\ell+1}$ for $\Gamma$ and $\bar\Gamma$, respectively. 
Then we have that 
$$\alpha(\bar \Gamma) = \omega( \Gamma)\le \tfrac{q^m+q^{\frac{m}{2}+\ell}}{q^{\frac{m}{2}+\ell}+1}\qquad 
\text{and}\qquad \omega(\bar \Gamma) = \alpha(\Gamma) \le \tfrac{q^m+q^{\frac{m}{2}-\ell}}{q^{\frac{m}{2}-\ell}+1}.$$ 
Thus, (\textit{d}) is a consequence of (\textit{c}) because of the general bound 
$\alpha(X) \chi(X) \ge n$, where $X$ is a 
graph or order $n$. If $\tfrac{1}{2}m_{\ell}$ is odd, $\Gamma$ contains the graph $\Gamma_{q,m}(\tfrac{m}{2})$ by 
Proposition~\ref{teo subg}, then $q^{\ell}+1\mid q^{\frac m2}+1$. Therefore, 
$\omega(\Gamma) = \alpha(\bar\Gamma) = \omega(\bar \Gamma) = \alpha(\Gamma) = \chi(\Gamma) = \chi(\bar \Gamma) = q^{\frac m2}$
 (see \cite{SS}).

For ($e$), it is a well-known fact that the vertex and edge connectivity 
of a strongly regular graph is equal to its regularity degree (see \cite{BH2}, \cite{BM}).
So, in particular we have that $\kappa(\Gamma)= \varepsilon(\Gamma) = k_\ell = \tfrac{q^m-1}{q^\ell+1}$ and $\kappa(\bar \Gamma)= \varepsilon(\bar \Gamma) = \bar k_\ell=q^\ell k_\ell$, as desired.


Now, for ($f$), $\theta_2(\Gamma)$ is by definition the minimum nonzero Laplacian eigenvalue of $\Gamma$. 
Recall that the Laplacian eigenvalues $\{\mu_i\}$ of a $k$-regular graph are related with the eigenvalues $\{\lambda_i\}$ by 
$\mu_i=k-\lambda_i$. Since $\Gamma$ is $k_{\ell}$-regular, then 
$\theta_2(\Gamma) = k_\ell - \upsilon_\ell$ if $\tfrac 12 m_{\ell}$ is odd and 
$\theta_2(\Gamma) = k_{\ell} - \mu_{\ell}$ if $\tfrac 12 m_\ell$ is even. On the other hand $\bar\Gamma$ is $q^{\ell}k_{\ell}$-regular, then $\theta_2(\Gamma) = q^{\ell}k_\ell - \bar\mu_\ell$ if $\tfrac 12 m_{\ell}$ is odd and 
$\theta_2(\Gamma) = q^{\ell}k_{\ell} - \bar\upsilon_{\ell}$ if $\tfrac 12 m_\ell$ is even.

Finally, to see (\textit{g}), given a graph $X$ the isoperimetric constant $h(X)$ satisfies (see \cite{Mo})
$$\tfrac 12 \theta_2(X) \le h(X) \le \sqrt{\theta_2(X) (2\Delta(X)-\theta_2(X))},$$ 
where $\Delta(X)$ is the maximal degree of $X$.
Since 
$\Delta(\Gamma) = k_{\ell}$ 
and $\Delta(\bar \Gamma) = q^{\ell}k_{\ell}$, 
the inequalities in the statement follows directly from the above ones and item ($f$), and the proof is complete.
\end{proof}

We now illustrate several previous results. 
\begin{exam}
The smallest graph in $\mathcal{G}_{q,m}$ having a connected subgraph ($\ell \ne \frac m2$) corresponds to 
$q=2$, $m=12$ and $\ell=1$, giving the pair of subgraphs 
$$\Gamma_{2,12}(3) \subset \Gamma_{2,12}(1) \qquad \text{and} \qquad \bar \Gamma_{2,12}(1) \subset \bar \Gamma_{2,12}(3).$$ 
In the following table we show their parameters as strongly regular graphs, the intersection arrays and their eigenvalues computed with Theorem \ref{Spec Gamma},  
Propositions \ref{Spec complement} and \ref{prop srg}, and Corollary \ref{coro array}.

\renewcommand*{\arraystretch}{1.15}
\begin{center}
\begin{tabular}{|l|l|c|lll|}
\hline
{graphs} 
& {srg parameters} & {intersection array} 
& & {eigenvalues} 
&  \\
\hline
$\Gamma_{2,12}(1)$ 			& $(4096, 1365, 440, 462)$ &  $(1365, 924; 1, 462)$ & $k_1 = 1365$, & $\upsilon_1 = 21$, & $\mu_1 = -43$ \\
\hline
$\Gamma_{2,12}(3)$ 			& $(4096, 455, 54, 50)$ 		&	$(455, 400; 1, 50)$ & $k_3= 455$, \; & $\upsilon_3= 7$, \; & $\mu_3= -57$      \\
\hline
$\bar \Gamma_{2,12}(1)$ & $(4096, 2730, 1826, 1086)$ & $(2730, 903; 1, 1804)$ & $\bar k_1 = 2730$, \; & $\bar \upsilon_1 = 42$, \; & $\bar \mu_1 = -22$ \\
\hline
$\bar \Gamma_{2,12}(3)$ & $(4096, 3640, 3234, 3240)$ & $(3640, 405; 1,  832)$ & $\bar k_3 = 3640$, \; & $\bar \upsilon_3 = -8$, \; & $\bar \mu_3 = 56$ \\
\hline
\end{tabular}
\end{center}
From the above table one can check that $924=2^1 \cdot 462$ and $400=2^3 \cdot 50$. Moreover, $k_i = (2^6+1) \upsilon_i$ for $i=1,3$ and that $k_3 \mid k_1$ and $\upsilon_3 \mid \upsilon_1$.
Also, $\bar k_1 = 2 k_1$, $\bar \mu_1 = 2 \mu_1$, $\bar \upsilon_1 +1= \tfrac{\mu_1+1}{2}$ and 
$\bar k_3 = 2^3 k_3$, $\bar \mu_3 = 2^3 \upsilon_3$, $\bar \upsilon_3 +1= \tfrac{\mu_3+1}{2^3}$.
\hfill $\lozenge$
\end{exam}

\section{Energy} \label{sec:6}
Another important spectral invariant of a graph is its energy. Given a graph $\Gamma$ of $n$ vertices with eigenvalues $\{\lambda_i\}_{i=1}^n$, 
the \textit{energy} of $\Gamma$ is defined by 
\begin{equation} \label{energy} 
E(\Gamma) = \sum_{i=1}^n |\lambda_i|.
\end{equation}

As a corollary of Theorem \ref{Spec Gamma}, we obtain the energy of the graphs $\Gamma_{q,m}(\ell)$ and $\bar \Gamma_{q,m}(\ell)$.
\begin{prop} \label{prop energy}
	The energy of $\Gamma=\Gamma_{q,m}(\ell) \in \mathcal{G}_{q,m}$ and $\bar \Gamma = \bar \Gamma_{q,m}(\ell) \in \mathcal{G}_{q,m}$ are given by  
\begin{align} \label{energys}
\begin{split}
E(\Gamma) 
& = \begin{cases} 
2k_\ell |\mu_\ell|, 						 & \quad  \text{ if $\frac 12 m_\ell$ is even}, \\[1.5mm]
2k_\ell q^\ell |\upsilon_\ell|,				 & \quad  \text{ if $\frac 12 m_\ell$ is odd},
\end{cases}  
\\[1.5mm]
E(\bar \Gamma) 
&= \begin{cases} 
2 \bar k_\ell |\bar \upsilon_\ell|, 		& \qquad  \text{ if $\frac 12 m_\ell$ is even}, \\[1.5mm]
2  \bar k_\ell |\upsilon_\ell|, 			& \qquad  \text{ if $\frac 12 m_\ell$ is odd}, 
\end{cases}
\end{split} 
\end{align}
for $\ell \ne \tfrac m2$ and 
\begin{equation} \label{energies m2}
E(\Gamma_{q,m}(\tfrac m2)) = E(\bar \Gamma_{q,m}(\tfrac m2)) =  2q^{\frac m2}(q^{\frac m2}-1) = 2 k_\ell q^\ell = 2\bar k_\ell,
\end{equation}
where $k_\ell$, $\upsilon_\ell$, $\mu_\ell$ and $\bar k_\ell$, $\bar \upsilon_\ell$, $\bar \mu_\ell$ 
are the eigenvalues of $\Gamma_{q,m}(\ell)$ and $\bar \Gamma_{q,m}(\ell)$ as given in 
Theorem~\ref{Spec Gamma} and Proposition \ref{Spec complement}, respectively. 
In particular, $2k\mid E(\Gamma)$ and $2\bar k \mid E(\bar \Gamma)$ for any $\ell$ and also
$\lambda_{min} \mid E(\Gamma)$ and $\bar \lambda_{min} \mid E(\bar \Gamma)$ for $\ell \ne \frac m2$, 
where $\lambda_{min}$, $\bar \lambda_{min}$ are the smallest eigenvalues of $\Gamma_{q,m}(\ell)$,
$\bar \Gamma_{q,m}(\ell)$ respectively. 
\end{prop}

\begin{proof}
By Theorem~\ref{Spec Gamma} we have that 
$Spec(\Gamma_{q,m}(\ell)) = \{ [k_\ell]^1, [\nu_\ell]^{q^\ell k_\ell}, [\mu_\ell]^{k_\ell} \}$ 
for $\ell\ne \frac m2$ with $k_\ell$, $\nu_\ell$ and $\mu_\ell$ as in \eqref{Autval} and 
$Spec(\Gamma_{q,m}(\frac m2)) = \{ [q^{\frac m2}-1]^{q^{\frac m2}}, [-1]^{q^{\frac m2}(q^{\frac m2}-1)} \}$.
By \eqref{energy} we have
	$$E(\Gamma_{q,m}(\ell)) = k_\ell + m(\upsilon_\ell) \, |\upsilon_\ell| + m(\mu_\ell) \, |\mu_\ell| \qquad \text{and} \qquad  
E(\Gamma_{q,m}(\tfrac m2)) = 2q^{\frac m2}(q^{\frac m2}-1)$$
for $\ell \ne \frac m2$. 
The expression for $E(\Gamma_{q,m}(\ell))$ with $\ell \ne \frac m2$ follows directly after straightforward but tedious computations from \eqref{Autval}, considering the different cases. To compute $E(\bar \Gamma_{q,m}(\ell))$ we proceed similarly using Proposition \ref{Spec complement}.
The divisibility of the energies of $\Gamma$ and $\bar \Gamma$ by $2k$ and $2\bar k$ respectively are obvious from \eqref{energys}. 

For the remaining assertion,  notice that by Notes \ref{note1} and \ref{note2} and Corollary \ref{prop eigen comps}, in case $\ell \ne \frac m2$ we have $E(\Gamma) = 2 k |\lambda_{min}|$ if $\tfrac 12 m_\ell$ is even and $E(\Gamma) = 2 \bar k |\lambda_{min}|$ if $\tfrac 12 m_\ell$ is odd, while $E(\bar \Gamma) = 2 \bar k |\bar \lambda_{min}|$, 
if $\tfrac 12 m_\ell$ is even and, $E(\Gamma) = 2 k |\bar \lambda_{min}|$ if $\tfrac 12 m_\ell$ is odd.
The result clearly follows from this observation.  
\end{proof}

\subsubsection*{Equienergetic non-isospectral graphs}
Two graphs of the same order are said to be \textit{equienergetic} if they have the same energy and \textit{isospectral} if they have the same spectrum.
Thus, by \eqref{energys}, no two of the graphs in $\mathcal{G}_q$ nor in $\bar{\mathcal{G}}_q$ are mutually equienergetic; 
that is 
$$E(\Gamma_{q,m}(\ell)) \ne E(\Gamma_{q,m}(\ell')) \qquad \text{and} \qquad E(\bar \Gamma_{q,m}(\ell)) \ne E(\bar \Gamma_{q,m}(\ell'))$$
for every fixed $q$, $m$ and every $\ell \ne \ell'$. 
By definition, isospectrality implies equienergeticity. The converse is false; the simplest counterexample 
is given by the pair $C_4$ and $K_2 \otimes K_2$, since $Spec(C_4) = \{[0]^2, [2]^2\}$ and 
$Spec(K_2 \otimes K_2) = \{[-1]^2, [1]^2 \}$, and hence $E(C_4)=E(K_2 \otimes K_2)=4$.
There are some more examples in the literature of equienergetic non-isospectral graphs. 
	
Equienergetic and non-isospectral Cayley graphs were studied for unitary Cayley graphs $G_R=X(R,R^*)$ over a finite commutative ring with identity (\cite{PV Equi1}) and for generalized Paley graphs (\cite{PV Equi3}, \cite{PV Equi2}).
As a consequence of the expressions for the energy in Proposition \ref{prop energy}, 
we now give some pairs of equienergetic and non-isospectral integral graphs obtained with our families.  

\begin{prop} \label{coro pairs}
We have the following pairs of equienergetic and non-isospectral graphs: 
	\begin{enumerate}[$(a)$] 
		\item $\Gamma_{q,m}(\ell)$ and $\bar \Gamma_{q,m}(\ell)$, for $\tfrac 12 m_\ell$ odd. \msk
		
		\item $\Gamma_{q,m}(\frac m2) = q^{\frac m2} K_{q^{\frac m2}}$ and 
		$\bar \Gamma_{q,m}(\frac m2) = K_{q^{\frac m2} \times q^{\frac m2}}$. \msk 
		 
		\item $\Gamma_{q,m}(\ell) \times K_2$ and $\Gamma_{q,m}(\ell) \otimes K_2$, for $\ell \ne \frac m2$. \msk 
		
		\item $\bar \Gamma_{q,m}(\ell) \times K_2$ and $\bar \Gamma_{q,m}(\ell) \otimes K_2$, for $\ell \ne \frac m2$.
	\end{enumerate}
\end{prop}

\begin{proof}
	The pairs of graphs in ($a$) and ($b$) are equienergetic by \eqref{energys}, since $\bar k = kq^\ell$, and 
	\eqref{energies m2} respectively. 
	The graphs in each equienergetic pair are non-isospectral to each other since $\Gamma_{q,m}(\ell)$ and $\bar \Gamma_{q,m}(\ell)$ have different spectra by Theorem \ref{Spec Gamma} and Proposition \ref{Spec complement}.
	
	Parts ($c$) and ($d$) are consequences of Theorem 8 in \cite{BVM}, which asserts that if $G$ is a connected graph with eigenvalues $\lambda_1,\ldots,\lambda_n$, then $G \times K_2$ and $G \otimes K_2$ are equienergetic and non-isospectral if and only if $|\lambda_i|\ge 1$ for $i=1,\ldots,n$. The graphs $\Gamma_{q,m}(\ell)$ and $\bar \Gamma_{q,m}(\ell)$ are integral. 
	For $\ell \ne \frac m2$, the graphs 
	$\Gamma_{q,m}(\ell)$ are connected with nonzero eigenvalues, hence ($c$) follows. The graphs $\bar \Gamma_{q,m}(\ell)$ are always connected and have non-trivial eigenvalues if $\ell \ne \frac m2$, this imply ($d$).
\end{proof}

Related to items ($a$) and ($b$), in \cite{PV Equi1} we have obtained pairs of integral equienergetic and non-isospectral graphs $\{G_R,\bar G_R\}$. On the other hand, in Theorem 5.4 (resp.\@ Proposition 6.6) in \cite{PV Equi2}, we have classified all strongly regular graphs (resp.\@ semiprimitive GP-graphs) which are equienergetic and non-isospectral with their complements. The graphs in item ($a$) (resp.\@ ($a$) and ($b$)) fall into this case. 
Also, in \cite{PV Equi3} we have obtained infinite pairs of integral graphs $\Gamma(3,q)$ and $\Gamma(4,q)$ as in \eqref{Gkq} which are equienergetic and non-isospectral with their complements in the non-semiprimitive case (in the semiprimitive case,  $\Gamma(3,p^m)=\Gamma_{2,m}(1)$ and $\Gamma(4,p^m)=\Gamma_{3,m}(1)$).

\subsubsection*{Hyperenergeticity and maximal energy}
It is known that $E(\Gamma) \ge 2\sqrt{n-1}$ with equality if and only if $\Gamma = K_{1,n-1}$, i.e.\@ for $1$-stars (\cite{Gu}).
A graph $\Gamma$ is called \textit{hyperenergetic} if 
$$E(\Gamma) > E(K_n) = 2(n-1).$$
We now show that all the graphs considered are generically hyperenergetic.

\begin{coro} \label{coro hyper}
	The graphs $\Gamma_{q,m}(\ell) \in \mathcal{G}_{q,m}$ and $\bar \Gamma_{q,m}(\ell) \in \bar{\mathcal{G}}_{q,m}$ 
	with $\ell \ne \frac m2$ are hyperenergetic, except for the Clebsch graph $\Gamma_{2,4}(1)$.
\end{coro}

\begin{proof}
	The result follows directly from \eqref{energys} using Theorem \ref{Spec Gamma} and Proposition~\ref{Spec complement}. In fact, 
	for $\ell \ne \frac m2$, one checks that 
	$E(\Gamma_{q,m}(\ell)) > 2(q^m-1)$ and $E(\bar \Gamma_{q,m}(\ell)) > 2(q^m-1)$ for any $(q,m,\ell) \ne (2,4,1)$. 
	(This is a direct proof. It is known that every SRG is hyperenergetic unless four cases (\cite{PM}); and out of these,
	only $srg(16,5,0,2)=\Gamma_{2,4}(1)$ is in our family.)
\end{proof}

Koolen and Moulton (\cite{KM}) gave two upper bounds for the energy of a graph $\Gamma$ with $n$ vertices and $e$ edges: 
one in terms of $n$ only
\begin{equation} \label{energy upper1}
E(\Gamma) \le \tfrac 12 n(1+\sqrt n);
\end{equation} 
and another one which, in the case that $\Gamma$ is $k$-regular, takes the following simple form
\begin{equation} \label{energy upper2}
E(\Gamma) \le k + \sqrt{k(n-1)(n-k)}.
\end{equation} 
A graph has \textit{maximal energy} 
if equality holds in \eqref{energy upper1} and we will say that it has \textit{$k$-maximal energy} 
if equality holds in \eqref{energy upper2}. 

\begin{coro} \label{comp clebsch}
	No graph in the families $\mathcal{G}_{q,m}$ and $\bar{\mathcal{G}}_{q,m}$ has maximal energy.
	The only $k$-maximal energetic graph in $\mathcal{G}_{q,m} \cup \bar{\mathcal{G}}_{q,m}$ is the complement of the Clebsch graph $\bar \Gamma_{2,4}(1)=srg(16,10,6,6)$.
\end{coro}

\begin{proof}
	The first assertion follows from Proposition \ref{prop energy} by a case-by-case analysis.
	For the second one, by Theorem 3 in \cite{KM}, equality holds in \eqref{energy upper1} if and only if $\Gamma$ is an SRG 
	with parameters	$srg(n, \frac{n+\sqrt n}2, \frac{n+\sqrt n}4, \frac{n+\sqrt n}4)$. By Remark \ref{inf fam}, the only possibility is $\bar \Gamma_{2,4}(1)=srg(16,10,6,6)$.
\end{proof}

Although the graphs in the families $\mathcal{G}_{q,m}$ and $\bar{\mathcal{G}}_{q,m}$ are not $k$-maximal energetic, we will show that they are \textit{asymptotically} $k$-maximal energetic. 

Denote by $E=E_0(\Gamma) = k + \sqrt{k(n-1)(n-k)}$ the right hand side of \eqref{energy upper2} and let
$\bar E = E_0(\bar \Gamma)$ 
be the corresponding one for the complement graph $\bar \Gamma$.
Balakrishnan proved (\cite{Bal}) that for each $\epsilon >0$, there exist infinitely many non-complete unitary Cayley graphs 
$\Gamma_n = X(\mathbb{Z}_n,\mathbb{Z}_n^*)$ such that $E(\Gamma_n)< \epsilon E_0$. We now show that there is an infinite family of generalized Paley graphs satisfying the same condition.

\begin{prop} \label{min Energy}
Given $\epsilon>0$, for fixed $m, \ell$ (resp.\@ $q$) there exist $q_0$ (resp.\@ $m_0$, $\ell_0$) 
such that for every $q\ge q_0$ (resp.\@ $m\ge m_0$, $\ell\ge \ell_0$) we have
$E(\Gamma_{q,m}(\ell)) < \epsilon E_0$ and $E(\bar \Gamma_{q,m}(\ell)) < \epsilon \bar E_0$.
\end{prop}

\begin{proof}
It is enough to show that $\lim\limits_{q \rightarrow \infty} \tfrac{E(\Gamma_{q,m}(\ell))}{E_0}=0$ and 
$\lim\limits_{m, \ell \rightarrow \infty} \tfrac{E(\Gamma_{q,m}(\ell))}{E_0}=0$, and similarly for the limits involving $\bar \Gamma_{q,m}(\ell)$.
Notice that $E_0=k_{\ell}+\sqrt{k_{\ell}(q^m-1)(q^m-k_{\ell})}$ where $k_{\ell}=\tfrac{q^m-1}{q^\ell+1}$, then we have 
	$$E_0=k_{\ell}(1+\sqrt{q^{m+\ell}+1}).$$ 
	By Theorem \ref{Spec Gamma} and Proposition \ref{prop energy} we have that    
$$ 	\tfrac{E(\Gamma_{q,m}(\ell))}{E_0} = \tfrac{2k_\ell |\mu_\ell|}{k_{\ell}(1+\sqrt{q^{m+\ell}+1})} = \tfrac{2(q^{\frac m2 + \ell}+1)}{(q^\ell+1)(1+\sqrt{q^{m+\ell}+1})}.$$
in the case $\tfrac 12 m_\ell$ even. 
Thus, 
$$\lim_{q \rightarrow \infty} \tfrac{E(\Gamma_{q,m}(\ell))}{E_0} =  \lim_{q \rightarrow \infty} \tfrac{1}{\sqrt{q^\ell}} =0$$
and similarly $\lim_{m, \ell \rightarrow \infty} \tfrac{E(\Gamma_{q,m}(\ell))}{E_0} =  \lim_{m, \ell \rightarrow \infty} q^{- \frac{\ell}{2}} =0$.
The case $\tfrac 12 m_\ell$ odd is analogous.  
For the complements we proceed in the same way, but we omit the details, and the proof is complete. 
\end{proof}

\section{The Waring's problem} \label{sec:8}
Here we give an application of our results to the Waring's problem over finite fields. The Waring's problem is about to compute or estimate the Waring number, which is the smallest number 
\begin{equation} \label{WN}
s=g(k,q)	
\end{equation}

such that the diagonal equation 
	$$x_1^k +\cdots + x_s^k = b$$ 
has a solution for every $b\in \ff_q$ with respect to $k$. In other words, determine the minimum number $s$ such that every element of $\ff_q$ can be written as a sum of at most $s$ terms of $k$-th powers in the field.  

The Waring number $g(k, p^n)$ exists if and only if $\frac{p^n-1}{p^d-1} \nmid k $ for all $d \mid n$, $d\ne n$. 
Also, since $g(k, q) = g(d, q)$ with $d = (k, q-1)$, it is enough to assume that $k\mid p^m-1$.
It is well known that $g(k, p) \le k$ and 
\begin{equation} \label{gkpk}
g(k, p) = k \qquad \Leftrightarrow \qquad k = 1, \, 2, \, \tfrac{p-1}2, \, p-1.
\end{equation}
Moreover, Small proved (see \cite{Sm}) that
\begin{equation} \label{small}
k\mid q-1 \qquad \text{and} \qquad 2 \le k < q^{\frac m4}+1 \qquad \Rightarrow \qquad g(k, q) = 2.
\end{equation}
To find the exact value of $g(k,q)$ is in general difficult and many upper bounds are known 
(see for instance \cite{MP}, or the survey \cite{Win}). The Waring numbers of arbitrary generalized Paley graphs were recently studied by the authors in \cite{PV} and \cite{PV2}.  

\begin{thm} \label{teo waring}
Let $q$ be a prime power and let $m,\ell$ be positive integers such that $\ell \mid m$ with $m_\ell$ even. 
If $\ell \neq \tfrac m2$, then  
$g(q^\ell+1,q^m) = 2$.
\end{thm}

\begin{proof}
By Proposition \ref{properties}, the graph $\Gamma_{q,m}(\ell)$ is connected since $\ell \neq \frac m2$. 
Moreover, $\Gamma_{q,m}(\ell)$ has diameter 2 since it is a strongly regular graph. 
Thus, if $a\in \ff_{q^m}^*$ then $1\le d(a,0) \le 2$. 
If $d(a,0)=1$, then $a0$ is an edge an hence $a-0 \in S_\ell$, i.e.\@ 
$a=x^{q^\ell+1}$ for some $x\in \ff_{q^m}^*$. 
On the other hand, if $d(a,0)=2$ then there is a path $ab0$, that is $a-b\in S_\ell$ and 
$b-0\in S_\ell$. In other words, there exist elements $x_1, x_2 \in \ff_{q^m}^*$ such that 
	$a-b=x_1^{q^\ell+1}$ 
	and $b=x_2^{q^\ell+1}$, 
from which we get that 
	$$a=x_1^{q^\ell+1} + x_2^{q^\ell+1}$$ 
and hence $g(q^\ell+1, q^m) = 2$, as desired.  
\end{proof}
Note that if $q$ is even and $m_\ell$ is odd, $\Gamma_{q,m}(\ell)$ is the complete graph $K_{q^m}$ which has diameter one. Hence, 
$g(k,q^m)=1$ in this case.

\begin{rem}
$(i)$ If $q=p^m$ with $m=4t$, taking $\ell = \frac{m}{4}$ we have that $m_{\ell}=4$ is even and thus
$g(p^{\frac{m}{4}}+1, q)=2$, by Theorem \ref{teo waring} since $p^{\frac m4}+1 \mid p^m+1$. 
In this way, we have extended a little bit Small's result \eqref{small}, that is $g(k,q)=2$ for any $2\le k\le q^{\frac 14}+1$.

\noindent ($ii$) 
In 2008, Moreno and Castro \cite{MC} showed that $g(k,p^{2\ell s})=2$ for any $k\mid p^{\ell}+1$ with $k<p^{\ell}+1$ and $s>1$.
Taking $q=p$ prime in Theorem \ref{teo waring} we obtain, in particular, the extreme case missing in Moreno-Castro's result, that is $g(p^{\ell}+1,p^{2\ell s}) = 2$ for any $s>1$.
\end{rem}

We now stress that there are generalized Paley graphs not in the families considered in this paper that are not strongly regular.
\begin{rem}
Consider the generalized Paley graph $\Gamma = GP(p^m,\frac{p^m-1}{k})$ with $k$ not necessarily equal to $p^\ell+1$ 
(see Remark \ref{GP}). By using the same argument as in the proof of Theorem \ref{teo waring}, one can show that if 
$\Gamma=srg(v,k',e,d)$ is strongly regular with $d>0$, then $g(k,p^m)=2$. Since it is known that in general $g(k,p^m)\ne 2$ 
(see Chapter 7 of \cite{MP}) then there must exist generalized Paley graphs which are not strongly regular. 

For instance, we have $g(\frac{p-1}2,p)=\frac{p-1}2$ by \eqref{gkpk}, hence $g(\frac{p-1}2,p) \ne 2$ 
for $p\ne 5$. In this case, 
the graph 
	$$GP(p,2)=X(\ff_p, \{x^{(p-1)/2} :x\ne 0 \})$$ 
is the cycle $C_p$ which is not strongly regular. 
For non-trivial examples of this kind we can take the pairs $(k,p) = (4,17),(5,31)$ with Waring number $g(k,p)=3$ 
(see the table in \S7.3.50 of \cite{MP}). 
This implies that the graphs 
	$$GP(17,4) = X(\ff_{17},\{x^4 : x\ne 0\}) \qquad \text{ and } \qquad GP(31,6) = X(\ff_{31},\{x^5:x\ne 0\})$$ 
are not strongly regular. 
\end{rem}

\section{Ramanujan graphs} \label{sec:9}
We recall that a $k$-regular graph $G$ is Ramanujan if for any eigenvalue $\lambda$ of $G$ it holds $\lambda=k$ or $|\lambda|\le 2\sqrt{k-1}$. It is very well-known that the complete graphs $K_n$ are Ramanujan for every $n$ and Paley graphs $P({q})$ are Ramanujan for any $q\equiv 1 \pmod 4$.

We will next show that some of the graphs $\Gamma_{q,m}(\ell)$ are Ramanujan (this can only happen in characteristics $2$ and $3$)  while all the complements $\bar \Gamma_{q,m}(\ell)$ are Ramanujan graphs.

\begin{thm} \label{Ram q23}
The graph $\Gamma_{q,m}(\ell)\in \mathcal{G}_{q,m}$ is Ramanujan if and only if 
$q\in \{2, 3, 4\}$, $\ell=1$ and $m\ge 4$. 
Furthermore, any graph $\bar\Gamma_{q,m}(\ell)\in \bar{\mathcal{G}}_{q,m} $ is Ramanujan.
\end{thm}

\begin{proof}
Suppose first that $\Gamma_{q,m}(\ell)$ is Ramanujan. 
Thus, since $m_\ell$ is even (hence $m$ is even), by Theorem~\ref{Spec Gamma} we have that $\# S_{\ell} = k_\ell$ and that the non-trivial eigenvalue with the biggest absolute value is $\mu_\ell$ given in \eqref{Autval}, that is $\lambda(\Gamma_{q,m}(\ell)) = |\mu_\ell|$. 
Then, they satisfy the inequality \eqref{rama} which now takes the form
$$\tfrac{q^{\frac{m}{2}+\ell} + \varepsilon}{q^{\ell}+1} \leq 2 \big(\tfrac{q^m-1}{q^{\ell}+1}-1 \big)^{\frac 12}$$
which is equivalent to 
\begin{equation} \label{desig raman}
(q^{\frac{m}{2}+\ell} + \varepsilon)^2\leq 4(q^{\ell}+1) (q^m-q^{\ell}-2).
\end{equation} 
From this we get 
$q^{\frac m2} (q^{\frac m2 +\ell} + 2\varepsilon q^\ell-4(q^\ell+1)q^{\frac m2}) \le -4(q^\ell+1)(q^\ell+2) <0$.
In particular, we must have 
	$$q^{\frac m2 +\ell} + 2\varepsilon q^\ell-4(q^\ell+1)q^{\frac m2}<0.$$
Taking $x=q^\ell$, this is equivalent to $x^2-2bx-4\le 0$ with $b=\varepsilon q^{-\frac m2}-2<0$. This implies that $x <-2b$, that is
$$q^\ell \le 4- \tfrac{2\varepsilon}{q^{\frac m2}}.$$ 
This can only happen when $q=2,3,4$ with $\ell=1$ or $q=2$ with $\ell=2$. But this last case cannot occur since if $q=2$ and $\ell=2$, the graph $\Gamma_{2,m}(2)$ satisfy \eqref{rama} only for $m=2,4$. But the graphs $\Gamma_{2,2}(2)$ and $\Gamma _{2,4}(2)$ are not allowed. 

Now, if $m=2$, the graphs $\Gamma_{2,2}(1)$, $\Gamma_{3,2}(1)$ and $\Gamma_{4,2}(1)$ are not Ramanujan. These graphs, by part ($b$) of Theorem \ref{Spec Gamma} are a disjoint union of complete graphs, hence not connected (although $K_n$ is Ramanujan), since $\ell=1=\tfrac m2$ in this case. 
Therefore, we must have $m \ge 4$ when $q=2,3,4$ with $\ell=1$.

\smallskip

Now we check that the graphs $\Gamma_{q,m}(1)$ with $q=2,3,4$ and $m\ge 4$ are indeed Ramanujan. 
Suppose that $q=2$. 
In this case, \eqref{desig raman} reads $(2^{\frac{m}{2}+1}+ \varepsilon)^2 \leq 12 (2^m-4)$,  
which in turn is equivalent to $2^{m+3}-\varepsilon 2^{\frac m2 +2}\ge 49$
which clearly holds for any $m\ge 4$.
Therefore $\Gamma_{2,2t}(1)$ is Ramanujan for all $t\ge 2$. 
Now, assume that $q=3$. 
In this case, \eqref{desig raman} reads
$$3^{m+2} + 2\varepsilon 3^{\frac{m}{2}+1}+1\leq 16(3^m-5)=3^{m+2}+7\cdot3^m-80,$$ 
which in turn gives
$7\cdot3^m - 2 \varepsilon 3^{\frac{m}{2}+1} \ge 81$, which holds 
for $m\geq 4$. Therefore $\Gamma_{3,2t}(1)$ is Ramanujan for all $t\ge 2$.
Finally, if $q=4$, \eqref{desig raman} takes the form 
$(4^{\frac m2+1}+\varepsilon)^2 \le 20(4^m-6)$, which is equivalent to $4^{m+3}-2\varepsilon 4^{\frac m2 +1}\ge 121$, which holds for every $m\ge 2$. However, $\Gamma_{4,2}(1)$ is not Ramanujan since it is not connected. Thus $\Gamma_{4,2t}(1)$ is Ramanujan for every $t\ge 2$, and the result thus follows.

\sk 

Finally, we show that any complementary graph $\bar \Gamma =\bar\Gamma_{q,m}(\ell)$ in the family $\bar{\mathcal{G}}_{q,m}$ is Ramanujan. We know that $\bar \Gamma$ is a connected $q^{\ell}k_{\ell}$-regular graph. By Proposition \ref{Spec complement}, the non-trivial eigenvalue of highest absolute value of $\Gamma$ is $\bar{\mu}_{\ell} = q^{\ell} \upsilon_{\ell}$. Then we just have to prove that \eqref{rama} holds, that is  
$$\tfrac{q^{\ell}(q^{\frac{m}{2}}-\varepsilon)}{q^{\ell}+1} 
\le 2 \big\{ \tfrac{q^{\ell}(q^m-1)}{q^{\ell}+1}-1 \big\}^{\frac 12}.$$
This inequality is equivalent to 
$q^{2\ell}(q^m-2\varepsilon q^{\frac{m}{2}}+1)\le 4(q^{\ell}+1)(q^{m+\ell}-2q^{\ell}-1)$,	
which in turn holds if and only if 
$$3q^{m+2\ell} +4q^{m+\ell} + 2\varepsilon q^{\frac{m}{2}+\ell} \ge 8q^{2\ell} +12q^{\ell} + 4.$$ 
This inequality can be further simplified to 
$$ q^{\frac m2} (3q^{\frac m2+\ell} +4q^{\frac m2} + 2\varepsilon) \ge 8 (q^{\ell} +2)$$ 
which is always true, because $\ell\ne \frac{m}{2}$ and $m\ge4$ in all these cases.
\end{proof}

\begin{rem}
We have shown that $\mathcal{G} = \{\mathcal{G}_{q,m}\}_{q,m}$ contains 3 infinite families of Ramanujan graphs, namely
\begin{equation} \label{familes234} 
\mathcal{G}_2 := \{ \Gamma_{2,2t}(1) \}_{t\ge 2}, \qquad \mathcal{G}_3 := \{ \Gamma_{3,2t}(1) \}_{t\ge 2}, 
\qquad \text{and} \qquad  \mathcal{G}_4 := \{\Gamma_{4,2t}(1) \}_{t\ge 2}.
\end{equation}
Thus, we have obtained infinite families of Ramanujan graphs defined over $\ff_2, \ff_3$ and $\ff_4$.

By Proposition \ref{teo subg}, $\Gamma_{q,m}(1) \subset \Gamma_{q,m'}(1)$ if and only if $m \mid m'$. In this way, 
we can construct several infinite towers of Ramanujan graphs in $\mathcal{G}_q = \{\mathcal{G}_{q,m}\}_m$ (i.e.\@ over $\ff_q$) 
with fixed $q=2,3$ or $4$, although with different regularity degree. 
For instance, for each fixed $q \in\{2,3,4\}$, we can take the following tower 
of $\frac{q^{2^n}-1}{4}$-regular graphs ($n\ge 2$)
$$\Gamma_{q,4}(1) \subset \Gamma_{q,8}(1) \subset \cdots \subset \Gamma_{q,2^n}(1) \subset \Gamma_{q,2^{n+1}}(1) \subset \cdots  $$  
\end{rem}

\begin{rem} 
The first graph in the family $\mathcal{G}_2$ is $\Gamma_{2,4}(1)=srg(16,5,0,2)$ known as the Clebsch graph. 
The first graph in the family $\mathcal{G}_3$, i.e.\@ $\Gamma_{3,4}(1)$, is by Theorem \ref{prop srg} 
an $(81,20,1,6)$ strongly regular graph. Brouwer and Haemers proved that this graph is unique (\cite{BH}). 
In particular, we have proved that the the Clebsch graph $srg(16,5,0,2)$ and the Brouwer-Haemers graph $\Gamma_{3,4}(1)=srg(81,20,1,6)$, as well as their complements, are all Ramanujan graphs. 
\end{rem}

We now give the parameters of the 3 families of Ramanujan graphs over $\ff_2$, $\ff_3$ and $\ff_4$ 
and their complementary families (using Theorem \ref{prop srg}); while for the first 3 graphs in each of the families we give their spectra (using Theorem \ref{Spec Gamma} and Proposition \ref{Spec complement}):

(\textit{a}) Families over $\ff_2$:
\begin{align*}
 & \Gamma_{2,2t}(1)      = srg(4^t, \tfrac{4^t-1}3, \tfrac{4^t + (-2)^{t+1} - 8}{9}, \tfrac{4^t + (-2)^{t} - 2}{9}) \\
 & \bar \Gamma_{2,2t}(1) = srg(4^t, \tfrac{2(4^t-1)}3, \tfrac{4^{t+1} + (-2)^{t} - 14}{9}, \tfrac{4^{t+1} + (-2)^{t+1} - 2}{9})
\end{align*}
\begin{center}
\renewcommand*{\arraystretch}{1.1}
\begin{table}[h!]
\caption{The families over $\ff_2$.}  \label{tabla f2}
\begin{tabular}{|c|l|l|c|}
\hline
$t$ 	& {graph} 				 
&     {srg parameters} 		  
& {spectrum} 
\\ \hline
2			& $			\Gamma_{2,4}(1)$ & $(16,5,0,2)$							& $\{ [5]^{1}, [1]^{10}, [-3]^{5} \}$ \\
2			& $\bar \Gamma_{2,4}(1)$ & $(16,10,6,6)$						& $\{ [10]^{1}, [2]^{5}, [-2]^{10} \}$ \\ \hline
3			& $			\Gamma_{2,6}(1)$ & $(64,21,8,6)$						& $\{ [21]^{1}, [5]^{21}, [-3]^{42} \}$ \\

3			& $\bar \Gamma_{2,6}(1)$ & $(64,42,26,30)$				  & $\{ [42]^{1}, [2]^{42}, [-6]^{21} \}$ \\ \hline
4			& $			\Gamma_{2,8}(1)$ & $(256,85,24,30)$					& $\{ [85]^{1}, [5]^{170}, [-11]^{85} \}$ \\
4			& $\bar \Gamma_{2,8}(1)$ & $(256,170,114,110)$			& $\{ [170]^{1}, [10]^{85}, [-6]^{170} \}$ \\ \hline
\end{tabular}
\end{table}
\end{center}

\renewcommand*{\arraystretch}{1.1}
(\textit{b}) Families over $\ff_3$:
\begin{align*}
 & \Gamma_{3,2t}(1)      = srg(9^t, \tfrac{9^t-1}4, \tfrac{9^t + 2 (-3)^{t+1} - 11}{16}, \tfrac{9^t + 2(-3)^{t} - 3}{16}) \\
 & \bar \Gamma_{3,2t}(1) = srg(9^t, \tfrac{3(9^t-1)}4, \tfrac{9^{t+1} + 2(-3)^{t} - 27}{16}, \tfrac{9^{t+1} + 2(-3)^{t+1} - 3}{16})
\end{align*}
\begin{center}
\begin{table}[h!]
\caption{The families over $\ff_3$.} \label{tabla f3}
\begin{tabular}{|c|l|l|c|}
\hline
$t$ 	& {graph} 
				 &    {srg parameters} 
				 		  & {spectrum} 
				 		  \\ \hline
2			& $			\Gamma_{3,4}(1)$ & $(81,20,1,6)$							& $\{ [20]^{1},  [2]^{60}, [-7]^{20} \}$ \\
2			& $\bar \Gamma_{3,4}(1)$ & $(81,60,45,42)$						& $\{ [60]^{1},  [6]^{20}, [-3]^{60} \}$ \\ \hline
3			& $			\Gamma_{3,6}(1)$ & $(729,182,55,42)$					& $\{ [182]^{1}, [20]^{182}, [-7]^{546} \}$ \\
3			& $\bar \Gamma_{3,6}(1)$ & $(729,546,405,420)$				& $\{ [546]^{1}, [6]^{546}, [-21]^{182}\}$ \\ \hline
4			& $			\Gamma_{3,8}(1)$ & $(6561,1640,379,420)$			& $\{ [1640]^{1}, [20]^{4920}, [-61]^{1640}\}$ \\
4			& $\bar \Gamma_{3,8}(1)$ & $(6561,4921,3699,3660)$		& $\{ [4921]^{1}, [60]^{1640}, [-21]^{4920}\}$ \\ \hline
\end{tabular}
\end{table}
\end{center}

(\textit{c}) Families over $\ff_4$:
\begin{align*}
 & \Gamma_{4,2t}(1)      = srg(16^t, \tfrac{16^t-1}5, \tfrac{16^t +3(-4)^{t+1} - 14}{25}, \tfrac{16^t + 3(-4)^{t} - 4}{25}) \\
 & \bar \Gamma_{4,2t}(1) = srg(16^t, \tfrac{4(16^t-1)}5, \tfrac{16^{t+1} + 3(-4)^{t} - 44}{25}, \tfrac{16^{t+1} + 3(-4)^{t+1} - 4}{25})
\end{align*}
\begin{center}
\begin{table}[h!]
\caption{The families over $\ff_4$.} \label{tabla f4}
\begin{tabular}{|c|l|l|c|}
\hline
$t$ 	& {graph} 				 
&      {srg parameters} 		  
& {spectrum} 
\\ \hline
2			& $			\Gamma_{4,4}(1)$ & $(256,51,2,12)$								& $\{ [51]^{1}, [3]^{204}, [-13]^{51} \}$ \\
2			& $\bar \Gamma_{4,4}(1)$ & $(256,204,164,156)$						& $\{ [204]^{1}, [12]^{51}, [-4]^{204} \}$ \\ \hline
3			& $			\Gamma_{4,6}(1)$ & $(4096,819,194,156)$						& $\{ [819]^{1}, [51]^{819}, [-13]^{3276} \}$ \\
3			& $\bar \Gamma_{4,6}(1)$ & $(4096,3276,2612,2652)$				& $\{ [3276]^{1}, [12]^{3276}, [-52]^{819} \}$ \\ \hline
4			& $			\Gamma_{4,8}(1)$ & $(65536,13107,2498,2652)$			& $\{ [13107]^{1}, [51]^{52428}, [-205]^{13107}\}$ \\
4			& $\bar \Gamma_{4,8}(1)$ & $(65536,52428,41972,41820)$		& $\{ [52428]^{1}, [204]^{13107}, [-52]^{52428}\}$ \\ \hline
\end{tabular}
\end{table}
\end{center}

The parameters of $\Gamma_{2,4}(1)$, $\bar \Gamma_{2,4}(1)$, $\Gamma_{2,6}(1)$, $\bar \Gamma_{2,6}(1)$, 
$\Gamma_{2,8}(1)$, $\bar \Gamma_{2,8}(1)$, $\Gamma_{3,4}(1)$, $\bar \Gamma_{3,4}(1)$, $\Gamma_{3,6}(1)$, $\bar \Gamma_{3,6}(1)$ and 
$\Gamma_{4,4}(1)$, $\bar \Gamma_{4,4}(1)$ are in coincidence with the parameters in the tables of Brouwer's web page (see \cite{BWP}). 
The remaining graphs, 
	$$\Gamma_{3,8}(1), \bar \Gamma_{3,8}(1), \qquad \Gamma_{4,6}(1), \bar \Gamma_{4,6}(1) \qquad \text{and} \qquad  \Gamma_{4,8}(1), \bar \Gamma_{4,8}(1),$$ 
are not in those tables.

\goodbreak

\subsubsection*{Ihara zeta function}
Ihara proved that a graph is Ramanujan if and only if the Ihara zeta function of $\Gamma$ satisfies the Riemann hypothesis in this context (\cite{Ih}). So, we now focus on these functions for our graphs.

\goodbreak 

The Ihara zeta function $\zeta_\Gamma(t)$ of a $k$-regular graph $\Gamma$ 
has a determinantial expression (\cite{H}) 
\begin{equation} \label{zeta1}
\zeta_\Gamma (t) = \frac{(1-t^2)^{1-r(\Gamma)}}{\det(Id -tA + (k-1)t^2 Id)}
\end{equation} 
where $r(\Gamma)$ is the circuit rank of $\Gamma$ defined by 
$r(\Gamma)=\chi(\Gamma)+c$, 
where $\chi(\Gamma)=e-n$ is the Euler characteristic ($e$ is the number of edges and $n$ the number of vertices, and $c$ the number of connected component of $\Gamma$) and 
$A$ is the adjacency matrix of $\Gamma$. 
Also, by \cite{Coo}, if $\Gamma$ is connected, non-bipartite and has minimal degree at least 2 then $1-t^2$ does not divide the determinant $\det(Id -tA- (k-1)t^2 Id)$.
By Ihara's result,  $\zeta_{\Gamma}(t)$ and $\zeta_{\bar \Gamma}(t)$ satisfy the Riemann hypothesis for all $\Gamma \in \mathcal{G}_2 \cup \mathcal{G}_3 \cup \mathcal{G}_4$ and all $\bar \Gamma=\bar \Gamma_{q,m}(\ell)$.

For a connected $k$-regular graph we have that $1-r(\Gamma) = n-e =n-\frac{nk}2 <0$. 
Further, $A$ is conjugated to the diagonal matrix $diag(\lambda_1, \lambda_2, \ldots, \lambda_n)$. Hence, if
	$$Q(s,t) = 1-st +(k-1)t^2 \in \Z[t],$$ 
then, by \eqref{zeta1}, the reciprocal of $\zeta_{\Gamma}(u)$ is an integral polynomial of degree $2e=kn$ given by
\begin{equation} \label{zeta2}
{\zeta_\Gamma (t)}^{-1} = (1-t^2)^{e-n} \prod_{1\le i \le n} Q(\lambda_i,t) \in \Z[t].
\end{equation} 
Note that $Q(k,1)=0$ and hence $Q(k,t)= (t-1)((k-1)t-1)$.
By the previous comments, the following result is automatic. 
\begin{prop} \label{ihara}
The Ihara zeta function of $\Gamma \in \mathcal{G}_{q,m} \cup \bar{\mathcal{G}}_{q,m}$ is given by
\begin{equation} \label{zeta3}
{\zeta_{\Gamma}(t)}^{-1} = (1-t^2)^{(\frac k2 -1)q^m} (t-1)((k-1)t-1) \, Q(\upsilon,t)^{m_\upsilon} \, Q(\mu,t)^{m_\mu}
\end{equation}
where $Spec(\Gamma) = \{ [k]^1, [\upsilon]^{m_\upsilon}, [\mu]^{m_\mu}\}$. 
\end{prop}

The \textit{complexity} $K(\Gamma)$ of a graph $\Gamma$ is the number of spanning trees of $\Gamma$. 
By Kirchhoff's formula, the complexity of the graphs $\Gamma \in \mathcal{G}_{q,m} \cup \bar{\mathcal{G}}_{q,m}$
is 
\begin{equation} \label{compl fla}
K(\Gamma) = \tfrac{1}{q^m} (k-\upsilon)^{m_\lambda}(k-\mu)^{m_\mu}.
\end{equation}
This invariant is also computable in terms of the Ihara zeta function of $\Gamma$.
By a result of Hashimoto (\cite{H}), we have 	
$$K(\Gamma) = \frac{1}{2^r \chi(\Gamma)}  \lim_{t \rightarrow 1} \frac{\zeta_\Gamma^{-1}(t)}{(1-t)^r}.$$
Thus, by \eqref{zeta3} and using that $Q(\lambda,1)= (k-\upsilon)$ and $Q(\mu,1)= (k-\mu)$ we get \eqref{compl fla}.

\begin{exam} \label{ex zetas}
	We now explicitly give the Ihara zeta functions for the Clebsch and the Brouwer-Haemers graphs and for their complements, and we also compute their complexities. We use Tables \ref{tabla f2} and \ref{tabla f3}.
	
	(\textit{i}) For the Clebsch graph $\Gamma = \Gamma_{2,4}(1)$ and its complement we have 
	\begin{align*}
	& \zeta_{\Gamma}^{-1}(t)  = (1-t^2)^{24} \, (1-5t+4t^2) \, (1-t+4t^2)^{10} \, (1+3t+4t^2)^{5}, \\[1mm] 
	& \zeta_{\bar \Gamma}^{-1}(t)  = (1-t^2)^{64} \, (1-10t+9t^2) \, (1-2t+9t^2)^{5} \, (1+2t+9t^2)^{10}.
	\end{align*}
	The complexities are given by
	$$	 K(\Gamma) = \tfrac{1}{16}(5-1)^{10}(5+3)^{5}= 2^{31}, \qquad 
	K(\bar \Gamma) = \tfrac{1}{16}(10-2)^{5}(10+2)^{10}= 2^{31} 3^{10}.$$ 

	(\textit{ii}) For the Brouwer-Haemers graph $\Gamma = \Gamma_{3,4}(1)$ and its complement we have
	\begin{align*}
	& \zeta_{\Gamma}^{-1}(t)      = (1-t^2)^{729} \, (1-20t+19t^2) \, (1-2t+19t^2)^{60} \, (1+7t+19t^2)^{20}, \\[1mm]
	& \zeta_{\bar \Gamma}^{-1}(t) = (1-t^2)^{2349} \, (1-60t+59t^2) \, (1-6t+59t^2)^{20} \, (1+3t+59t^2)^{60}.
	\end{align*}
	Thus, the complexities of the graphs are 
	$$	 K(\Gamma) = \tfrac{1}{81} 18^{60} 13^{20}= 2^{60} 3^{116} 13^{20}, \qquad 
K(\bar \Gamma) = \tfrac{1}{81} 54^{20} 63^{60}= 2^{20} 3^{174} 7^{60}.$$

\end{exam}

\end{document}